\documentclass[a4paper]{scrartcl}

\usepackage{float}
\usepackage[hidelinks]{hyperref}
\usepackage{microtype}
\usepackage{amsmath,amssymb,amsthm}
\usepackage{mathtools}
\usepackage{booktabs}
\usepackage{graphicx}
\usepackage{tikz}
\usepackage{pgfplots}
\usepackage{pgfplotstable}
\usepackage[capitalize]{cleveref}

\newtheorem{theorem}{Theorem}
\newtheorem{corollary}{Corollary}
\newtheorem{lemma}{Lemma}
\newtheorem{remark}{Remark}

\crefname{section}{Section}{Sections}
\crefname{subsection}{Section}{Sections}
\crefname{equation}{}{}

\usepgfplotslibrary{groupplots}
\usepgfplotslibrary{colorbrewer}
\pgfplotsset{
    compat=1.16,
    every tick label/.append style={font=\small},
    legend style={
        font=\small,
        /tikz/every even column/.append style={column sep=0.2cm},
    },
    cycle list/Set1,
    cycle multiindex* list={
        mark list*\nextlist
        Set1\nextlist
    },
    every axis plot/.append style={semithick, mark options={scale=0.75}},
}

\newcommand{\dtau}{\,\textnormal{d}\tau}
\newcommand{\N}{\mathbb{N}}
\newcommand{\R}{\mathbb{R}}
\DeclareMathOperator{\rank}{rank}
\DeclareMathOperator{\im}{img}
\DeclareMathOperator{\Span}{span}

\begin{document}

\title{A Parareal Algorithm with Low-Rank Coarse Solvers\thanks{
Funded by the Deutsche Forschungsgemeinschaft
(DFG, German Research Foundation) under Germany's Excellence Strategy
EXC 2044 –390685587, Mathematics Münster: Dynamics–Geometry–Structure
and the Swiss National Science Foundation.
This work has been supported by the Ministry of Culture and Science NRW as part of the Lamarr Fellow Network.
}}

\author{%
    Martin J. Gander\thanks{Université de Gen\`eve, Switzerland (martin.gander@unige.ch).}
    \and Mario Ohlberger\thanks{Mathematics Münster, Germany
        (mario.ohlberge@uni-muenster.de, stephan.rave@uni-muenster.de).}
    \and Stephan Rave\footnotemark[3]}

\date{May 29, 2026}

\maketitle

\begin{abstract}
    We consider a new class of Parareal algorithms, which use ideas
    from localized reduced basis methods to construct the coarse solver
    from truncated SVD approximations of the transfer
    operators mapping initial values for a given time interval to the solution 
    at the end of the interval.  By leveraging randomized singular value
    decompositions, these low-rank approximations are obtained
    embarrassingly parallel by computing local fine solutions for
    random initial values. We show a priori and a posteriori error bounds in terms
    of the computed singular values of the transfer operators.  Our
    numerical experiments demonstrate that our approach can
    significantly outperform Parareal with single-step coarse solvers.
    At the same time, it permits to further increase parallelism in
    Parareal by trading global iterations for a larger number of independent
    local solves.
\end{abstract}

\section{Introduction}

Solving time-dependent partial differential equations (PDEs) 
efficiently is a crucial challenge in many scientific and 
engineering applications. 
Traditional time integration methods can be computationally 
expensive, particularly for complex models and long simulation times. 
The Parareal algorithm \cite{lions2001resolution} offers a compelling
approach to parallelizing the solution of time-dependent PDEs.  
By combining coarse-grained and
fine-grained parallelism, it provides significant speedups compared to
traditional serial or spatially-parallel methods, particularly for
problems with many time steps.  The core principle of
Parareal involves decomposing the time domain into a coarse temporal
grid, solving with a coarse solver on that grid, and then
iteratively refining the solution using local fine-scale computations
that can be computed in parallel without communication.  The
efficiency and accuracy of Parareal are critically dependent on the
performance of the coarse solver used to advance the solution on the
coarse grid, see, e.g., \cite{gander2007analysis,gander2008nonlinear,MR3803287,MR4876550}.

We present and analyze a new approach 
to improve the performance of Parareal by using localized-in-time reduced
basis techniques for the definition of the coarse solver.
More precisely, we consider as coarse solvers low-rank approximations of
the transfer operators mapping arbitrary initial values for a given time interval to the corresponding solution 
at the end of the interval.
Using randomized singular value decompositions (SVDs), we can efficiently construct
approximations of these transfer operators in an embarrassingly parallel manner.
Our approach is mainly targeted at linear (systems of) PDEs of parabolic type, for which the transfer operators
exhibit a rapid singular value decay, which means that small truncation ranks will be sufficient to obtain a good approximation.
This includes diffusion-advection-reaction equations with a sufficiently large diffusion term.
Other applications might include time-domain simulations of large, asymptotically stable ODE systems.
As the construction of the coarse solvers is independent of the inhomogeneous part of the PDE,
the same coarse solvers can be reused for new source terms without relevant additional computational work
(see \cref{rem:affine_g_n}).
Our method can be easily implemented on top of existing PDE solvers, as long as the solver allows
evaluating the adjoint transfer operator for the given PDE.

We refer to \cite{MR3701994,MR3408061,MR3379913} for a review on reduced basis methods 
and to \cite{Buhr2020245} concerning localized reduced basis constructions. 
Concepts based on the approximation of transfer operators using randomized 
SVDs have been studied both with respect to localization 
in space \cite{MR3824169,MR4401794} and localization in time \cite{MR4589122}. 

The Parareal algorithm typically employs a coarser time discretization 
as its coarse solver, leaving the spatial component unchanged 
from the fine-scale integrator.
In that case, the convergence of the Parareal algorithm is well
understood, showing superlinear rates with rapid convergence for parabolic
problems, and convergence difficulties for hyperbolic problems
\cite{gander2007analysis, gander2008nonlinear, gander2014analysis, MR4732743}.

Our approach diverges from this standard practice by using
coarse solvers that rely on low-dimensional, problem-adapted
approximations in space, but have the same accuracy in time as the fine solver.
A first Parareal variant using reduced basis models for the coarse solver 
was introduced in \cite{he2010reduced}. However, no practically feasible algorithm for constructing the reduced
basis was provided.
In \cite{chen2014use}, using ideas from the Krylov subspace Parareal approach 
\cite{FarhatCortialEtAl2006TimeparallelImplicitIntegrators,GanderPetcu2008AnalysisKrylovSubspace},
a global reduced basis was constructed on the fly from the linear span of the fine solver 
values obtained during the iteration of the Parareal algorithm. 
Our approach differs from \cite{chen2014use} as it computes near-best approximations of the fine
solver for arbitrary initial values instead of only considering the convergence history.
Further, in contrast to \cite{chen2014use}, our coarse solvers are constructed in an embarrassingly
parallel manner before any Parareal iterations take place.
Our construction is entirely local and does not require any communication between nodes associated with
different Parareal time intervals.

Regarding further Parareal methods with coarse solvers that use reduced approximations in space,
see \cite{MR4234221}, where reduced models obtained from asymptotic expansions for highly oscillatory 
differential equations were considered.
For problems with matrix-valued solutions that can be approximated by low-rank matrices, a Parareal method was
introduced in \cite{carrel2023low}, which uses dynamical low-rank approximations of different
ranks for the fine and coarse solvers.
We also refer to our recent discussion of the Parareal algorithm without coarse
solver, which includes a convergence analysis based on separation of variables 
and Fourier analysis \cite{GanderOhlbergerEtAl2024PararealAlgorithmCoarse}.

This paper is organized as follows:
In \cref{sec:algorithm} we define our new Parareal algorithm with low-rank coarse solvers
and discuss the algorithmic aspects of the construction of the coarse solvers via randomized SVDs.
We provide a complete convergence analysis for our algorithm in \cref{sec:algorithm},
demonstrating superlinear and linear convergence rates depending on the singular value spectra
of the transfer operators.
We also provide an easily computable a posteriori error bound.
In \cref{sec:numexp}, we finally illustrate our theoretical findings 
with a series of numerical 
experiments on one-, two-, and three-dimensional parabolic problems, 
demonstrating the practical efficiency and robustness of the 
proposed approach.

\section{Definition of the Algorithm}\label{sec:algorithm}

Consider a partition $0 = t_0 < t_1 < \ldots < t_N = T$ of the time
interval $[0, T]$ and a Hilbert space $V$.  Let $u_0 \in V$.  For each
$1 \leq n \leq N$, let an operator $F_n: V \to V$ be given.  We assume
that $F_n = F(\cdot, t_{n-1}, t_{n})$ is a fine solver for a given
linear system of differential equations with solution $u(t) \in V$, $t
\in [0,T]$. In other words, we assume that $u_n \approx u(t_n)$,
where the sequence $u_n \in V$ is given by
\[
    u_{n+1} := F_{n+1} u_{n}, \qquad 0 \leq n < N.
\]
We further assume the fine solver to be sufficiently good so that our
only goal is to accurately approximate the sequence $u_n$.  To
that end, we define coarse solvers $G_n: V \to V$, $1 \leq n \leq
N$, and consider Parareal algorithms of the form
\begin{equation}
    \label{eq:def_Parareal}
    u^{k+1}_{n+1} := F_{n+1} u^k_n + G_{n+1} u^{k+1}_n - G_{n+1} u^k_n, \qquad 0  \leq n < N, k \in \N_0,
\end{equation}
where $u^k_0 := u_0$ for all $k \in \N_0$ and
\[
    u^0_{n+1} := G_{n+1} u^0_n, \qquad 0  \leq n < N.
\]
Our goal is to find suitable coarse solvers $G_n$ such that $u^{k}_n$
quickly converges to $u_n$ as the number of Parareal iterations $k$
increases.  At the same time, each $G_n$ should be fast to evaluate,
since for each Parareal iteration \cref{eq:def_Parareal}, only the
fine solvers $F_n$ can be evaluated in parallel, whereas the $G_n$
have to be evaluated sequentially.  A typical choice for $G_n$ would
be a time stepping scheme with a coarse time-step size.

In this work, we consider $G_n$ which are obtained from a low-rank
approximation of the corresponding fine solver $F_n$.
To do so, we will always assume that each $F_n$ is a continuous,
affine linear operator.  In particular, we can write each $F_n$ as
\begin{equation}
    \label{eq:F_n_affine}
    F_n v = F_n'v + b_n, \qquad b_n := F_n 0,
\end{equation}
where $F_n': V \to V$ is a continuous linear operator.  Finally, we
assume that each $F_n'$ is a compact operator, which is always the
case in a discrete setting with finite-dimensional $V$.  In
particular, this means that $F_n'$ has a singular value decomposition
of the form
\begin{equation}
    \label{eq:svd}
    F_n' v = \sum_{r=1}^{\rank F_n'} \psi_{n,r} \cdot \sigma_{n,r} \cdot (\varphi_{n,r}, v)_V .
\end{equation}
Here, $(\cdot, \cdot)_V$ denotes the inner product on $V$,
$\sigma_{n,1} > \sigma_{n,2} > \ldots > 0$ are the uniquely defined
singular values of $F_n'$, $\varphi_{n,r} \in V$, $\|\varphi_{n,r}\| =
1$ are the corresponding right- and $\psi_{n,r} \in V$, $\|\psi_{n,r}\| =
1$ the corresponding left-singular vectors.  We note that when $V$ is
finite dimensional, \cref{eq:svd} is the matrix-free formulation of
the generalized SVD of $F_n'$ that takes the inner product on $V$ into
account.

We now define the coarse solver $G_n$ by considering a truncated
singular value decomposition of $F_n'$ and adding the same affine part
$b_n$:
\begin{equation}
    \label{eq:spectral_G}
    G_n v := \sum_{r=1}^{R_n} \psi_{n,r} \cdot \sigma_{n,r} \cdot (\phi_{n,r}, v) + b_{n},
\end{equation}
where $0 \leq R_n \leq \rank F_n'$ is the freely choosable truncation
rank.

Note that
\[
    (F_n - G_n) v = \sum_{r=R_n+1}^{\rank F_n'} \sigma_{n,r} \cdot (\phi_{n,r}, v) \cdot \psi_{n,r}
\]
is linear with
\begin{equation}
    \label{eq:truncation_error}
    \|F_n - G_n\| = \sigma_{n, R_n+1}.
\end{equation}
In particular, $G_n$ can be made an arbitrarily good approximation of
$F_n$ by increasing $R_n$.  We are mostly interested in the case of
linear parabolic PDEs, where the analytical solution operators for the
initial-value problems on each time interval can be expected to have rapidly
decaying singular values.
We expect $F_n'$ to be a sufficiently good approximation of the corresponding analytical
solution operator in the sense that it inherits its singular value decay
(see \cref{sec:effect_time_disc} for further discussion),
so small values of $R_n$ will be sufficient.  Indeed, in
\cite{GanderOhlbergerEtAl2024PararealAlgorithmCoarse} we have shown
that for the heat equation with homogeneous Dirichlet conditions in
one spatial dimension, the Parareal iteration \cref{eq:def_Parareal}
even converges for $G_n = 0$.  We have further shown the exponential decay
of the singular values of $F_n'$ when $F_n$ is chosen to be the exact
solution of the heat equation (see also experiment~1 below).  The
argument can be easily extended to higher spatial dimensions.  In the
context of model order reduction, the maps $F_n'$ are called transfer
operators in time.  In \cite{MR4589122} the compactness of $F_n'$ was
shown for a large class of parabolic PDEs.

\begin{remark}
    While assuming that the fine solvers $F_n$ are affine restricts
    the class of possible time integrators, we note
    that this condition is satisfied by all Runge-Kutta schemes,
    as long as the underlying system of differential equations is linear.
    Further, even when the time discretization leads to non-affine
    $F_n$, these $F_n$ still approximate the affine exact solution operator of the underlying
    system of differential equations, and we can interpret the
    randomized SVD algorithm discussed below as a numerical
    approximation of the truncated SVD of the exact solution operator.
    Thus, the assumption that the $F_n$ are affine is mainly a simplification for the theoretical
    analysis. In particular, the error bounds in
    \cref{sec:convergence_analysis} could be extended to also take
    the error between the fine solver and the exact solution
    operator into account.
\end{remark}

\begin{remark}
    \label{rem:affine_g_n}
    Note that the Parareal iteration \cref{eq:def_Parareal} can also
    be written as $u_{n+1}^{k+1} = F_{n+1}u_n^k + G_{n+1}'u_n^{k+1} -
    G_{n+1}'u_n^k$, where $G_n'$ is the linear part of $G_n$.  Thus,
    adding $b_n$ in the definition of $G_n$ only has an effect in
    the initialization phase where the $u_n^0$ are computed.  We
    choose to include $b_n$ in the definition of $G_n$ as this will
    give us a much better initial approximation of the solution
    without any additional sequential computations. 
\end{remark}

\subsection{Computation of $G_n$}

Even when $V$ is finite dimensional, computing a truncated SVD of
$F_n'$ using direct matrix-based algorithms is not an option:
determining the matrix of $F_n'$ would mean to locally solve the PDE
for every degree of freedom of the solution space.  Instead, we
consider the following basic randomized SVD algorithm to compute
$G_n$:
\begin{enumerate}
    \item Draw $R_n + p$ random vectors $\omega_i \in V$ according to
      some appropriate probability distribution on $V$, and compute
      $F_n' \omega_i$.  

      If we draw enough random vectors, $W:=
      \operatorname{span} \{F_n'\omega_1, \ldots,
      F_n'\omega_{R_n+p}\}$ will be a good approximation of the image
      of $F_n'$ with high probability.  In particular, we can
      approximate the SVD of $F_n'$ by computing the SVD of $P_WF_n'$,
      where $P_W$ is the $V$-orthogonal projection onto~$W$.
    \item Compute an orthonormal basis $w_1, \ldots, w_{R_n+p}$ of
      $W$, e.g., by using the modified Gram-Schmidt process.
    \item Compute $F_n'^*w_i$, $1 \leq i \leq R_n+p$, where $F_n'^*: V
      \to V$ is the adjoint of $F_n'$ given by $(F_n'^*w, v)_V = (w,
      F_n'v)_V$ for all $v,w \in V$.
    \item Compute an orthonormal basis $v_1, \ldots, v_{R_n+p}$ for
        $\operatorname{span} \{F_n'^*w_1, \ldots, F_n'^*w_{R_n+p}\}$.

        Then we have:
        \begin{align*}
            P_WF_n' v &= \sum_{i=1}^{R_n+p} w_i \cdot (w_i, F_n'v)_V
                      = \sum_{i,j=1}^{R_n+p} w_i \cdot (F_n'^*w_i, v_j)_V \cdot (v_j, v)_V.
        \end{align*}
    \item Compute a singular value decomposition of the
      small-dimensional matrix $M \in R^{(R_n+p) \times (R_n+p)}$,
      $M_{i,j} := (F_n'^*w_i, v_j)_V$.
      
      Since the orthonormal bases $w_i$ and
      $v_j$ span $\im F_n'$ and $(\ker
      F_n')^\perp$, it is easy to show that the singular values of $M$
      agree with the singular values of $P_WF_n'$.  Further, if
      $\underline{\psi}_r,\underline{\varphi}_r \in \R^{R_n+p}$ are
      the $r$-th left and right-singular vectors of $M$, then
      \[
        \psi_r := \sum_{i=1}^{R_n+p} \underline{\psi}_{r,i} \cdot v_i, \qquad
        \varphi_r := \sum_{i=1}^{R_n+p} \underline{\varphi}_{r,i} \cdot w_i
      \]
      are the $r$-th left and right-singular vectors of $P_WF_n'$.
    \item Return the first $R_n$ singular values and vectors $\sigma_r$,
      $\varphi_r$, $\psi_r$, $1 \leq r \leq R_n$.
\end{enumerate}\bigskip

The $p$ additionally drawn random vectors are called oversampling
vectors, and $p$ has to be chosen large enough to ensure that the
first $R_n$ singular values and vectors of $P_WF_n'$ are close to the
first $R_n$ singular values and vectors of $F_n'$.  In the case of $V =
\R^m$, extensive analysis is available that permits controlling the
approximation error with arbitrarily high probability \cite{MR2806637}.  In
particular, adaptive algorithms are available that choose $p$
a~posteriori, based on a prescribed (typically small) failure
probability.  An analysis for general finite-dimensional Hilbert
spaces can be found in~\cite{MR3824169}.

Randomized SVD algorithms perform well when the
singular values decay rapidly.  For operators with slow singular value
decay, algorithms with power iterations are considered where $W :=
\Span \{(F_n'F_n'^*)^qF_n'\omega_i \,|\, 1 \leq i \leq R_n+p\}$ with $q
\geq 1$.  For parabolic problems, the singular value decay is
generally fast enough such that power iterations are not necessary.
In fact, for the numerical examples considered in \cref{sec:numexp},
already a single oversampling vector ($p = 1$) turns out to be
sufficient.

\subsection{Computational costs}
Compared to classical coarse solvers, using low-rank coarse solvers
involves an additional setup phase where the truncated singular value
decompositions of the $F_n'$ as well as the $b_n$ have to be computed.  To
compute a $G_n'$ of rank $R_n$, $R_n$ evaluations of $F_n$
are required, followed by $R_n$ evaluations of $F_n'^*$.
The costs for computing the randomized SVD from these evaluations
are negligible compared to evaluating the fine solvers.
Additionally, a single evaluation of $F_n$ is required to obtain $b_n$.
Overall, after $k$ Parareal iterations, which each require a further fine
solver evaluation,
\begin{equation}\label{eq:num_F_eval}
    K_{F,n} := 2(R_n + p) + k + 1
\end{equation}
fine solver evaluations per time interval are required, compared to only $k$
evaluations for Parareal with classical coarse solvers.
However, since our coarse solvers lead to faster converging methods,
a smaller number $k$ of sequential Parareal iterations is required, whereas all
$2(R_n + p) + 1$ fine solver evaluations for computing $G_n$ can be
carried out in parallel for all time intervals
$[t_{n-1}, t_{n}]$ without communication.  Furthermore, the individual evaluations of
$F_n$, $F_n'$ and $F_n'^*$ can be carried out in parallel as well.
Hence, choosing $R_n$ allows us to freely balance the
amount of parallel work with the number of required communication events.
Moreover, since the same initial value problem is solved locally for $R_n + p$
initial values at the same time, our approach can benefit immensely from using direct linear solvers
or SIMD vectorization (see discussion in \cref{sec:ex4}).
Parallelism can be further increased at the expense of numerical accuracy
by evaluating $F_n'^*$ on an additional set of random vectors
instead of evaluating it on $w_1, \ldots, w_{R_n+p}$ (e.g., \cite[section 15]{MartinssonTropp2020RandomizedNumericalLinear}).
Finally, we will see in \cref{sec:ex4} that, in some cases, the convergence of our method
can be so much faster that $K_{F,n}$ can actually be smaller than the number
of Parareal iterations required using classical coarse solvers.

After the setup phase, an evaluation of
$G_n$ only requires the computation of $R_n$ inner products and the
linear combination of $R_n + 1$ vectors.
Thus, evaluating $G_n$ is typically much cheaper than using a
single backward Euler step as coarse solver.

\section{Convergence analysis}\label{sec:convergence_analysis}

We now show a priori and a posteriori error bounds for our algorithm.
For the a priori estimates, we follow the analysis approach in
\cite{gander2008nonlinear} under the assumption that we are solving a
linear system of differential
equations, see also \cite{MR4643841} and the recent book
\cite{GanderLunet2024} for more details.
Similar a priori error bounds for time-independent, jointly diagonizable
$F' = F_n'$, $G' = G_n'$ were obtained in \cite{DobrevKolevEtAl2017TwoLevelConvergenceTheory}
for the MGRIT method.
Our analysis is based on the following abstract error estimate
which holds for arbitrary Parareal algorithms with affine $F_n$ and $G_n$.
As before, we denote the linear parts of $F_n$ and $G_n$
by $F_n'$ and $G_n'$.

\begin{lemma}\label{Lemma1}
    \label{thm:a_priori_lemma}
    Assume that $\|F'_n - G'_n\| \leq \varepsilon$ and $\|G'_n\| \leq
    \delta$ for all $1 \leq n \leq N$.  Let $e_n^k := u_n^k - u_n$
    denote the Parareal error at time index $n$ and iteration $k$.
    Then we have
    \begin{equation}
        \label{eq:base_error_inequality}
        \|e_{n}^{k}\| \leq \varepsilon \|e_{n-1}^{k-1}\| + \delta \|e_{n-1}^{k}\|
    \end{equation}
    for all $n, k \geq 1$.
    Further, we have
    \begin{equation}
        \label{eq:base_error_inequality2}
        \|e_n^{k}\| \leq \varepsilon \sum_{m=1}^{n-1}\delta^{n-m-1}\|e_{m}^{k-1}\|.
    \end{equation}
\end{lemma}
\begin{proof}
    We have
    \begin{equation}
        \label{eq:base_error_identity}
        \begin{aligned}
            e_{n}^{k}
                &= (F_{n}u_{n-1}^{k-1} + G_{n}u_{n-1}^{k} - G_{n}u_{n-1}^{k-1}) - F_{n}u_{n-1} \\
                &= F_{n}'e_{n-1}^{k-1} + G_{n}'e_{n-1}^{k} - G_{n}'e_{n-1}^{k-1} \\
                &= (F_{n}'-G_{n}')e_{n-1}^{k-1} + G_{n}'e_{n-1}^{k}.
        \end{aligned}
    \end{equation}
    Hence, \cref{eq:base_error_inequality} follows from the triangle
    inequality and the bounds assumed on the norms of $F_{n}' -
    G_{n}'$ and $G_{n}'$.  We show \cref{eq:base_error_inequality2} by
    induction on $n$.  For $n=1$, \cref{eq:base_error_inequality}
    yields $\|e_1^k\| = 0$, since $\|e_0^{k-1}\| = \|e_0^k\| = 0$.  For
    the induction step, we use \cref{eq:base_error_inequality} again
    to obtain
    \begin{align*}
        \| e_{n}^{k} \| &\leq \varepsilon \|e_{n-1}^{k-1} \| + \delta \|e_{n-1}^{k} \|\\
            &\leq \varepsilon \|e_{n-1}^{k-1}\| + \delta \sum_{m=1}^{n-1-1} \delta^{n-1-m-1} \|e_m^{k-1}\|
            = \varepsilon \sum_{m=1}^{n-1} \delta^{n-m-1} \|e_m^{k-1}\|.
    \end{align*}
\end{proof}

Using the inequalities from Lemma \ref{Lemma1}, we obtain the
following superlinear and linear convergence results:
\begin{theorem}
    \label{thm:general_convergence}
    Let
    \begin{equation}
        \label{eq:bounds_delta_eps}
        \delta := \max_{1 \leq n \leq N} \sigma_{n,1}, \qquad
        \varepsilon := \max_{1 \leq n \leq N} \sigma_{n, R_n+1}.
    \end{equation}
    Then for any $1 \leq n \leq N$ we have
    \begin{equation}
        \label{eq:error_bound_initial}
        \|e_n^0\| \leq
            \sum_{m=0}^{n-1} \min\Bigl(2\delta, (n-m)\varepsilon\Bigr) \delta^{n-m-1}\|b_m\|,
    \end{equation}
    where $b_0 := u_0$.
    Furthermore, for any $1 \leq n \leq N$ and $k \in \N$ we have
    \begin{align}
        \label{eq:error_bound_eps_delta}
        \|e_n^k\| &\leq \hphantom{2}\varepsilon^k
                        \sum_{m=1}^{\hspace{0.5em}n-k\hspace{0.5em}} \binom{n-m}{k-1} \delta^{n-m-k} \|e_m^0\|\\
        \label{eq:error_bound_eps_delta2}
                  &\leq 2\varepsilon^k \sum_{m=0}^{n-k-1} \binom{n-m}{k} \delta^{n-m-k} \|b_m\|.
    \end{align}
    If $\delta \leq 1$, we have for $k \in \N$ that
    \begin{equation}
        \label{eq:error_bound_eps}
        \max_{1 \leq n \leq N}\|e_n^k\|
            \leq \binom{N}{k} \varepsilon^k \max_{1 \leq m \leq N-k} \|e_m^0\|.
    \end{equation}
    If $\delta < 1$, we have for $k \in \N$ that
    \begin{equation}
        \label{eq:error_bound_eps_inf_time}
        \max_{1 \leq n \leq N}\|e_n^k\|
        \leq \left(\frac{\varepsilon}{1 - \delta}\right)^k\max_{1 \leq n \leq N-k} \|e_n^0\|.
    \end{equation}
\end{theorem}
\begin{proof}
    We have $\|G_n'\| = \|F_n'\| = \sigma_{n,1}$ and \cref{eq:truncation_error}.
    In particular, the assumptions of \cref{thm:a_priori_lemma} are satisfied.

    To show \cref{eq:error_bound_initial}, note that
    \begin{align*}
        u_{n}^0 - u_{n} 
            &= \left(\sum_{m=0}^{n-1} G_{n}'\cdots G_{m+1}'b_m + b_{n}\right)
              - \left(\sum_{m=0}^{n-1} F_{n}'\cdots F_{m+1}'b_m + b_{n}\right) \\
            &= \sum_{m=0}^{n-1} G_n'\cdots G_{m+1}'b_m - F_n'\cdots F_{m+1}'b_m.
    \end{align*}
    Each of the terms in the sum can be bounded by
    $2\delta^{n-m}\|b_m\|$ using the triangle inequality and
    that $\|G_n'\|,\|F_n'\| \leq \delta$.  Alternatively, we get the bound
    $(n-m)\varepsilon\delta^{n-m-1}\|b_m\|$ by induction on~$n$,
    \begin{align*}
        &\|G_n'\cdots G_{m+1}'b_m - F_n' \cdots F_{m+1}'b_m\| \\
        &\hspace{10em}\leq \|(G_n'-F_n')G_{n-1}'\cdots G_{m+1}'b_m\| \\
        &\hspace{10em}\qquad\qquad + \|F_n'(G_{n-1}'\cdots G_{m+1}' - F_{n-1}'\cdots F_{m+1}')b_m\|\\
        &\hspace{10em}\leq \varepsilon\delta^{n-m-1}\|b_m\| 
            + \delta\cdot (n-1-m)\varepsilon\delta^{n-1-m-1}\|b_m\| \\
        &\hspace{10em}= (n-m)\varepsilon\delta^{n-m-1}\|b_m\|.
    \end{align*}
    Next, we show \eqref{eq:error_bound_eps_delta} by induction on
    $k$.  For $k=1$, \cref{eq:error_bound_eps_delta} is precisely
    \cref{eq:base_error_inequality2}.  For the induction step, we use
    \cref{eq:base_error_inequality2} again to obtain
    \begin{align*}
        \|e_n^{k+1}\|
        &\leq \varepsilon \sum_{m=1}^{n-1}\delta^{n-m-1}\|e_{m}^k\|\\
        &\leq \varepsilon \sum_{m=1}^{n-1}\delta^{n-m-1}
            \varepsilon^k \sum_{l=1}^{m-k} \binom{m-l}{k-1} \delta^{m-l-k} \|e_l^0\| \\
        &=\varepsilon^{k+1}\sum_{l=1}^{n-1-k}\underbrace{\sum_{m=l+k}^{n-1}\binom{m-l}{k-1}}_{\leq\binom{n-l}{k}}\delta^{n-l-k-1}\|e_l^0\|.
    \end{align*}
    The bound on the sum of binomial coefficients follows from the hockey-stick identity.
    The bound \cref{eq:error_bound_eps_delta2} follows by combining
    \cref{eq:error_bound_eps_delta} with
    \cref{eq:error_bound_initial},
    \begin{align*}
        \|e_n^k \|
            & \leq \varepsilon^k \sum_{m=1}^{n-k}
                    \binom{n-m}{k-1} \delta^{n-m-k}
                    \sum_{l=0}^{m-1} 2\delta^{m-l}\|b_l\|\\
            & \leq 2 \varepsilon^k \sum_{l=0}^{n-k-1}
                \underbrace{\sum_{m=l+1}^{n-k}\binom{n-m}{k-1}}_{\leq \binom{n-l}{k}} \delta^{n-k-l} \|b_l\|.
    \end{align*}
    If $\delta \leq 1$, the bound \eqref{eq:error_bound_eps} follows
    directly from \cref{eq:error_bound_eps_delta},
    \begin{align*}
        \|e_n^k\| \leq \varepsilon^k\max_{1 \leq m \leq n-k} \|e_m^0\| \sum_{m=1}^{n-k} \binom{n-m}{k-1}
        \leq\varepsilon^k\max_{1 \leq m \leq n-k} \|e_m^0\|\cdot\binom{n}{k},
    \end{align*}
    for each $1 \leq n \leq N$.  Thus, taking the maximum w.r.t.\ $n$
    and using $\binom{n}{k} \leq \binom{N}{k}$ yields the claim.

    Finally, to show \cref{eq:error_bound_eps_inf_time}, note that for
    $\delta < 1$, \cref{eq:base_error_inequality2} can be further
    bounded by
    \[
        \|e_n^{k}\| \leq \varepsilon \sum_{m=1}^{n-1}\delta^{n-m-1}\|e_{m}^{k-1}\| 
        <  \varepsilon \max_{0 \leq m \leq n-1} \|e_m^{k-1}\| \frac{1}{1-\delta}.
    \]
    Applying this inequality $k$ times yields the claim.
\end{proof}

We see that for stable fine solvers (in the sense that
$\delta \leq 1$), we obtain from
\cref{thm:general_convergence} similar superlinear and linear
convergence results as in \cite{gander2007analysis}. 
These assumptions are satisfied, for instance, for the heat equation or
finite-element discretizations thereof, when $F$ is
the exact solution operator in time or the result of an appropriate stable
time-discretization.  More precisely, we have $\delta <
1$ when (homogeneous) Dirichlet boundaries are present and $\delta =
1$ in the case of pure Neumann boundaries.
For $\delta < 1$, \cref{eq:error_bound_eps_delta,eq:error_bound_eps_delta2} show
that the influence of the error or source term contributions associated with time index $m$ 
on the error at time index $n$ decreases exponentially as $n-m$ increases.
We further note that the bound \cref{eq:error_bound_eps_delta2} does not depend on $T$
or the number of time intervals $N$.
Since the construction of the coarse solvers $G_n$ is entirely local, we obtain weak scalability
in the sense that the truncation rank $R$ does not need to be increased to maintain a constant
convergence rate as $N$ grows proportionally with $T$.

As an important special case, we consider the case of a
time-independent system with fixed fine solver $F'$, where
$G'$ is obtained by projecting $F'$ onto an invariant subspace.
In that case, we have the following exact representation of the error:
\begin{theorem}
    \label{thm:diagonal_convergence}
    Assume that $F' = F'_n$, $G' = G'_n$ for all $n \in \N$ and that
    $P$ is a projection (continuous, linear, idempotent operator) with
    $PF' = F'P$ and $G' = PF'P$.  Then for all $n, k \geq 0$ we have
    \begin{equation}
        e_n^k = \begin{cases}
            (F'-G')^ke_{n-k}^0 & n > k, \\
            0                & \text{otherwise}.
         \end{cases}
    \end{equation}
    Assuming $\|F' - G'\| \leq \varepsilon$, we obtain the improved estimate
    \begin{equation}
        \max_{1 \leq n \leq N}\|e_n^k\| \leq \varepsilon^k \max_{1 \leq n \leq N-k} \|e_n^0\|.
    \end{equation}
\end{theorem}
\begin{proof}
    The proof is the same as for the heat equation in
    \cite{GanderOhlbergerEtAl2024PararealAlgorithmCoarse}.
    Multiplying the error identity \cref{eq:base_error_identity} from the left with $P$
    yields
    \[
        Pe_{n}^{k} = P(F' - PF'P)e_{n-1}^{k-1} + P(PF'P)e_{n-1}^{k} = G'Pe_{n-1}^{k}.
    \]
    Hence, $Pe_{n}^{k} = (G')^{n}Pe_0^{k} = 0$.  Thus,
    multiplying \cref{eq:base_error_identity} by $I-P$ and using
    $(I-P)G' = (I-P)(PF'P) = PF'P - PF'P = 0$ yields
    \[
        e_{n}^{k} = (I-P)e_{n}^{k} = (I-P)(F'-G')e_{n-1}^{k-1} = (F'-G')(I-P)e_{n-1}^{k-1}
        = (F'-G')e_{n-1}^{k-1}.
    \]
    Recursively applying this identity yields $0$ when the time index reaches 0
    before the iteration count.  Otherwise, we get $e_n^k = (F'-G')^ke_{n-k}^0.$
\end{proof}

When $F'$ is a normal operator on a Hilbert space and $P$ an
orthogonal projection, $\|F'-G'\|$ is given by the maximum absolute
value of the neglected spectrum.  In the general (possibly) infinite-dimensional
case, we can formulate this as follows:
\begin{corollary}
    Assume that $F' = F'_n$, $G' = G'_n$ for all $n \in \N$.
    Assume that $F'$ is normal with spectral decomposition
    \[
        F' = \int_{\sigma(F')}\lambda d\mu_{F'}(\lambda), \qquad\text{and let}\qquad
        G' = \int_\Gamma \lambda d\mu_{F'}(\lambda),
    \]
    for some Borel measurable $\Gamma \subseteq \sigma(F')$.
    Then we have
    \begin{equation}
        \max_{1 \leq n \leq N}\|e_n^k\|
            \leq \Big(\sup_{\lambda \in \sigma(F')\setminus \Gamma} |\lambda|\Big)^k
                \cdot\max_{1 \leq n \leq N-k} \|e_n^0\|.
    \end{equation}
\end{corollary}
In particular, when $F'$ is self-adjoint, the SVD of $F'$ agrees with its spectral decomposition, and we
obtain:
\begin{corollary}
    \label{thm:spectral_bound}
    Assume that $F' = F'_n$, $n \in \N$, is self-adjoint and $G$ the low-rank coarse solver for truncation
    rank $R$.
    Then
    \begin{equation}
        \max_{1 \leq n \leq N}\|e_n^k\|
            \leq \sigma_{R + 1}^k \cdot\max_{1 \leq n \leq N-k} \|e_n^0\|.
    \end{equation}
\end{corollary}

To calculate the error bound \cref{eq:error_bound_eps_delta2},
only the quantities \cref{eq:bounds_delta_eps}
and $\|b_m\|$ are required, which are all known after the setup phase of the algorithm. 
However, while
\cref{eq:error_bound_eps_delta2} provides some insight into the
expected convergence speed of the Parareal algorithm, it significantly overestimates
the error in practice.
Thus, we further derive an a posteriori error
bound, which takes the current Parareal updates $u_n^{k} - u_n^{k-1}$
into account:

\begin{theorem}
    With the assumptions of \cref{thm:general_convergence}, for any $1
    \leq n \leq N$ and $k \in \N$ we have the a posteriori error bound
    \begin{equation}
        \label{eq:apost_bound}
        \|e_n^k\|
            \leq \varepsilon \sum_{m=1}^{n-1} \delta^{n-m-1} \|u^k_m - u^{k-1}_m\|.
    \end{equation}
    If $\delta < 1$, we additionally have
    \begin{equation}
        \label{eq:apost_bound_sup}
        \max_{1 \leq n \leq N}\|e^k_n\| \leq \frac{\varepsilon}{1-\delta} \cdot \max_{1 \leq n \leq N}\|u^k_n - u^{k-1}_n\|.
    \end{equation}
\end{theorem}
\begin{proof}
    We use the error identity
    \begin{align*}
        e_{n}^{k}
        &= u_n^{k} - F_nu_{n-1}^k + (F_nu_{n-1}^k - F_{n}u_{n-1}) \\
        &= F_nu_{n-1}^{k-1} + G_nu_{n-1}^k - G_nu_{n-1}^{k-1} - F_nu_{n-1}^k + F_n'e_{n-1}^k \\
        &= (F_n'-G_n')(u_{n-1}^{k-1} - u_{n-1}^k) + F_n'e_{n-1}^k.
    \end{align*}
    Taking norms then leads to
    \[
        \|e_n^k\| \leq \varepsilon \| u_{n-1}^k - u_{n-1}^{k-1}\| + \delta \|e_{n-1}^{k}\|.
    \]
    Applying this inequality repeatedly to its own right-hand side and noting that $e_0^k = u_0^k -
    u_0^{k-1} = 0$, we obtain \cref{eq:apost_bound}.
    For $\delta < 1$, we obtain \cref{eq:apost_bound_sup} from \cref{eq:apost_bound} using
    $\sum_{m=1}^{n-1}\delta^{n-m-1} \leq 1/(1-\delta)$.
\end{proof}
As for the a priori bound \cref{eq:error_bound_eps_delta2}, the quantities $\delta$ and $\varepsilon$
are known after the setup phase.
Thus, to evaluate \cref{eq:apost_bound,eq:apost_bound_sup} while the algorithm converges,
only the norms of the Parareal updates have to be computed.

\begin{figure}[t]
    \begin{center}
        \includegraphics{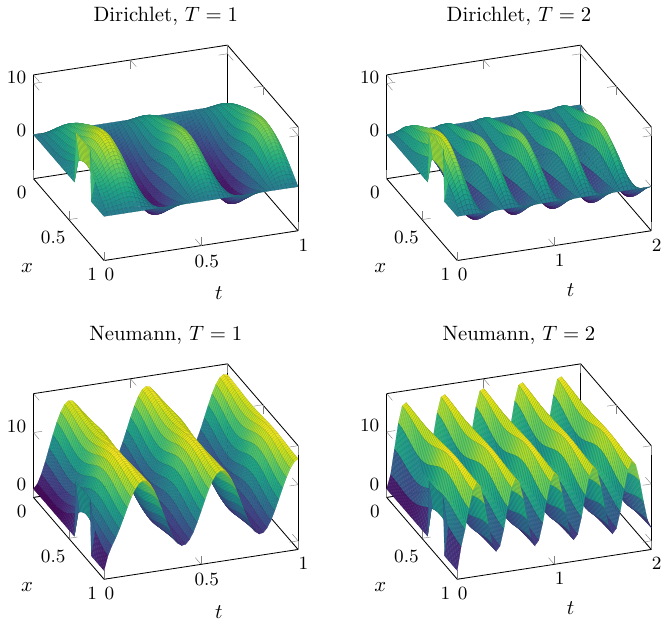}
    \end{center}
    \caption{%
      Experiment 1: Solutions of test problem \cref{eq:oned_heat} with
      different final times $T$ and boundary conditions.  }
    \label{fig:ex1_solution}
\end{figure}

\section{Numerical experiments}\label{sec:numexp}

We study the efficiency of the proposed low-rank coarse solvers in several
experiments of increasing complexity. In all cases, we use linear
finite elements for the spatial discretization and backward Euler
time-stepping with equidistant time-step size as the fine solver $F$.
The error between the fully discrete solution and its Parareal
approximation is measured in the Euclidean norm of the solution degrees of freedom
vectors (denoted as $\ell^2$-error).
Within each Parareal time-interval $[t_{n-1}, t_n)$, we approximate the solution
by considering the trajectory obtained from evaluating the fine solver
$F_n$ on $u_{n-1}^k$.

Our software implementation is based on the
\texttt{pyMOR}\footnote{\url{https://pymor.org}} model order reduction
library \cite{MilkRaveEtAl2016PyMORGenericAlgorithms}.  In particular,
we use \texttt{pyMOR}'s randomized SVD algorithm with no power
iterations and a variable number $p$ of oversampling vectors to
compute the low-rank coarse solvers.  We will compare the proposed
randomized SVDs with ``exact'' SVDs computed using
\texttt{SciPy}'s\footnote{\url{https://scipy.org}}
\cite{2020SciPy-NMeth} \texttt{ARPACK}
\cite{LehoucqSorensenEtAl1998ARPACKUsersGuide} wrapper.  In all cases,
we use a fixed truncation rank $R$ for all $G_n'$.  For the one- and
two-dimensional test-cases, we use \texttt{pyMOR}'s built-in
discretization toolkit, whereas experiment~4 uses
\texttt{scikit-fem}\footnote{\url{https://scikit-fem.readthedocs.io}}
\cite{GustafssonMcBain2020ScikitfemPythonPackage} for the
three-dimensional finite-element discretization.  The Parareal
algorithm is implemented with distributed-memory parallelization using
\texttt{mpi4py}\footnote{\url{https://mpi4py.readthedocs.io/}}
\cite{DalcinFang2021Mpi4pyStatusUpdate}.  All code and data needed to
conduct the numerical experiments is available on GitHub\footnote{\url{https://github.com/sdrave/LowRankParareal/tree/v2}}
and Zenodo \cite{ZenodoRecord}.

All experiments were executed on a dual-socket server equipped with two Intel
Xeon~Gold~6254 CPUs with 18 physical cores each.
A single thread per \texttt{MPI} rank was used for all
computations.  The server was running under \texttt{Ubuntu}~22.04,
using \texttt{CPython} 3.11.10 from the
\texttt{python-build-standalone}\footnote{\url{https://github.com/astral-sh/python-build-standalone}}
project.  All \texttt{Python} packages were obtained from
PyPI\footnote{\url{https://pypi.org}} in the exact versions specified
by the \texttt{uv.lock}\footnote{\url{https://docs.astral.sh/uv/}}
file contained in \cite{ZenodoRecord}.

\begin{figure}[t]
    \begin{center}
        \includegraphics{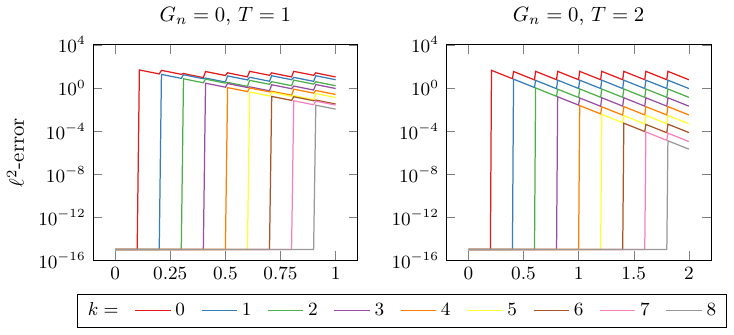}
    \end{center}
    \caption{%
      Experiment 1 with homogeneous Dirichlet conditions:
      $\ell^2$-errors of Parareal approximations over time for an
      increasing number of Parareal iterations; $G_n = 0$, 10 Parareal
      subintervals of different length by varying the final time $T$.
    }
    \label{fig:ex1_dirichlet_zero_G}
\end{figure}

\subsection{Experiment 1}

As the most basic test case, we consider a one-dimensional heat
equation on the unit interval,
\begin{equation}
    \label{eq:oned_heat}
    \begin{aligned}
        u_t(x,t) - u_{xx}(x,t) &= f(x,t) := 100\cdot\sin(5\pi t)\cdot(1+\cos(3\pi x)) ,\\
        u(x, 0) &= u_0(x) = 10 \chi_{[0.6, 0.8]},
    \end{aligned}
\end{equation}
$x \in (0,1)$, $t\in [0,T]$ for varying final time $T$.
If we apply homogeneous Dirichlet boundary conditions
\[
    u(0, t) = 0, \qquad u(1, t) = 0,
\]
then from Fourier analysis we know that the analytical solution of \cref{eq:oned_heat} is given by
\begin{equation}
    \label{eq:fourier_solution_formula}
    \begin{aligned}
        u(x, t) &= \sum_{m=1}^\infty \hat u_m(t)\sqrt{2}\sin(m\pi x) \\
        \hat u_m(t) &= \hat u_{0,m}e^{-m^2\pi^2t} + \int_0^t \hat f_m(\tau) e^{-m^2\pi^2(t-\tau)} \dtau,
    \end{aligned}
\end{equation}
where $\hat f_m$ and $\hat u_{0, m}$ denote the coefficients of the
Fourier sine series for $f$ and $u_0$.  Discretizing using a uniform
mesh with 100 elements in space and $100\cdot T$ time steps, we obtain
the solutions shown in the top row of \cref{fig:ex1_solution}.

\begin{figure}[t]
    \begin{center}
        \includegraphics{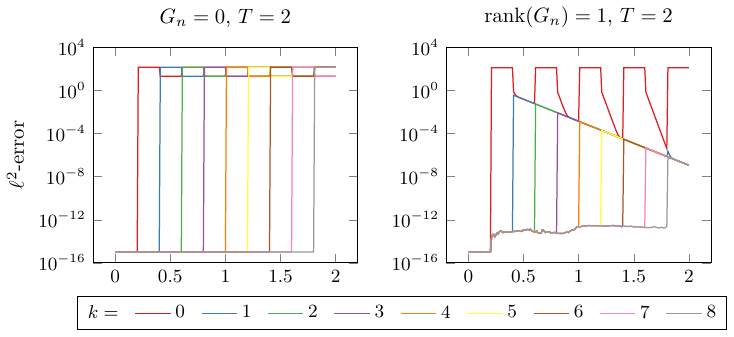}
    \end{center}
    \caption{%
      Experiment 1 with homogeneous Neumann conditions:
      $\ell^2$-errors of Parareal approximations over time for an
      increasing number of Parareal iterations;
      left: $G_n=0$, right: rank-1 $G_n$ with constant Fourier mode.
    }
    \label{fig:ex1_neu}
\end{figure}

\begin{figure}[t]
    \begin{center}
        \includegraphics{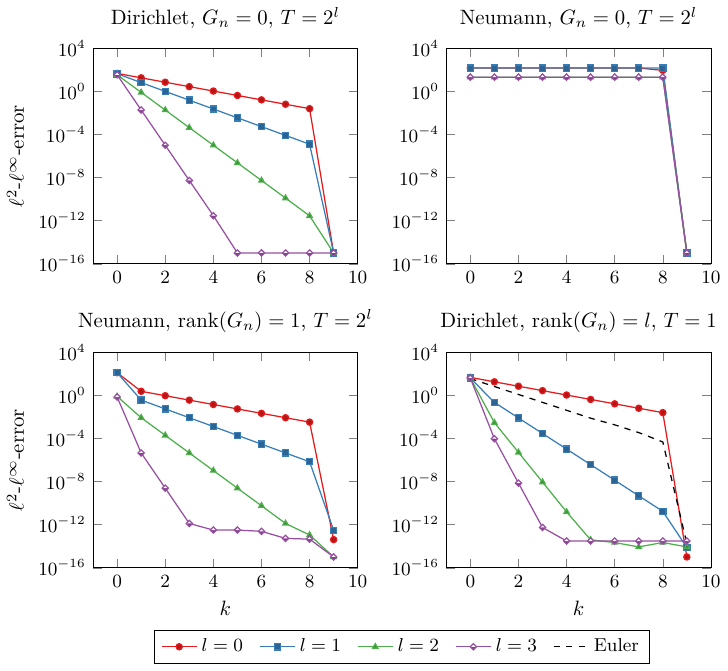}
    \end{center}
    \caption{%
      Experiment 1 with different boundary conditions and coarse
      solvers $G_n$.  Maximum $\ell^2$-errors over time vs.\ number
      of Parareal iterations for different final times $T = 2^l$ (top
      left/right, bottom left) and ranks $R$ of $G_n$ (bottom right).
      The dashed line in the bottom-right plot indicates the Parareal
      errors when the $G_n$ are given by a single backward Euler step.
    }
    \label{fig:ex1_max_errs}
\end{figure}

Now, for each $T$ we partition the time domain $[0,T]$ into 10
subintervals of the same length, resulting in $10\cdot T$ time steps per interval.
From \cref{eq:fourier_solution_formula},
we can expect that the linear part $F'$ of the fine
solver has an exponentially decaying spectrum and that the decay rate
increases with the length of the Parareal time intervals.  As all
Fourier modes decay exponentially over time, we expect
convergence even for $G = 0$, similar to the results in
\cite{GanderOhlbergerEtAl2024PararealAlgorithmCoarse}.
Indeed, in \cref{fig:ex1_dirichlet_zero_G},
which shows the error of the Parareal
iterates over time, we observe that the Parareal algorithm converges with
a rate that rapidly accelerates with increasing interval length.
The $\ell^2$-in-space $l^\infty$-in-time error is shown in the top-left plot of
\cref{fig:ex1_max_errs}.

If we replace the Dirichlet boundary conditions by Neumann boundary conditions,
\[
    u_x(0, t) = 0 \qquad u_x(1, t) = 0,
\]
we obtain a constant Fourier mode, which is not damped over time.
Hence, an error in the mean of the solution is propagated over the
entire time interval $[0, T]$, and simply choosing $G=0$ fails
miserably (\cref{fig:ex1_neu}, left and \cref{fig:ex1_max_errs},
top right).  To account for the constant mode, we now consider our
first low-rank coarse solvers of the form \cref{eq:spectral_G}
with the a priori choices $R_n = 1$, $\sigma_{n,1} = 1$ and
$\varphi_{n,1}=\psi_{n,1}$ being the finite-element coefficient vector
corresponding to the constant function with value 1.  In order to
correctly project onto the constant Fourier mode, we choose the
$L^2$-inner product in the definition of the $G_n$
\cref{eq:spectral_G}.  To make the results more comparable with the
case $G_n = 0$, we set $b_n = 0$ for all $n$.

As shown in the right plot of \cref{fig:ex1_neu} and the bottom-left plot of \cref{fig:ex1_max_errs}, we
successfully recover the convergence of the Parareal algorithm as observed in the Dirichlet case.

\begin{figure}[t]
    \begin{center}
        \includegraphics{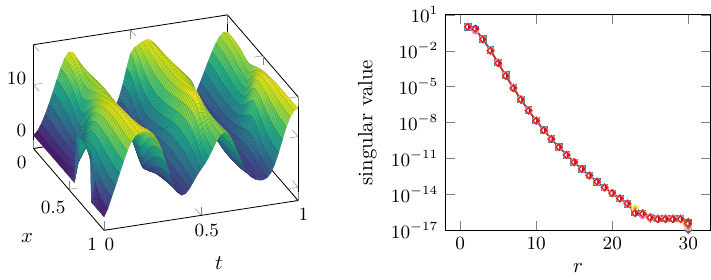}
    \end{center}
    \caption{%
        Experiment 2: Solution of test problem \cref{eq:ex2} (left) and singular values of the fine
        solvers $F_n'$ (right).
        The singular values only differ slightly with respect to the chosen time interval.
    }
    \label{fig:ex2_solution_svals}
\end{figure}

Returning to Dirichlet conditions, we finally study the effect of including higher-order Fourier modes
$\sqrt{2}\sin(m\pi x)$ into the definition of $G$ (again setting $b_n = 0$).
In the bottom-right plot of \cref{fig:ex1_max_errs}, we see that already
two Fourier modes suffice to let Parareal converge to machine precision after only five iterations.
This observed exponential convergence of our low-rank Parareal algorithm w.r.t.\ the rank of $G_n$ agrees
with \cref{thm:spectral_bound} if we take into account that by \cref{eq:fourier_solution_formula}, we have
\[
    \sigma_{R+1} = e^{-(R+1)^2\pi^2\Delta T}.
\]
A direct analysis of this case is contained in \cite{GanderOhlbergerEtAl2024PararealAlgorithmCoarse}.
We stress that already with a single Fourier mode, our algorithm converges much faster than a classical
Parareal method using a single backward Euler step as coarse solver (dashed line in
\cref{fig:ex1_max_errs}).

\subsection{Experiment 2}
In experiment~1, we considered the special case where all fine solvers
$F_n$ are equal and known analytically (up to discretization).  We now
lift both these assumptions by considering a one-dimensional heat
equation with a heat conductivity that varies both in space and time:
\begin{equation}
    \label{eq:ex2}
    \begin{gathered}
        u_t(x,t) - \partial_x[(1 + 0.9\sin(7\pi t + 2\pi x))u_{x}(x,t)] = f(x,t) ,\\
        u(x, 0) = u_0(x), \qquad u_x(0, t) = 0, \qquad u_x(1, t) = 0.
    \end{gathered}
\end{equation}
Here, the source term $f$ is again given by \cref{eq:oned_heat}, $x
\in (0,1)$ and $t \in [0,1]$.  As before, we discretize the problem
with a uniform mesh with 100 elements in space and $100$ time steps.
The solution is shown in \cref{fig:ex2_solution_svals} on the left.

\begin{figure}[t]
    \begin{center}
        \includegraphics{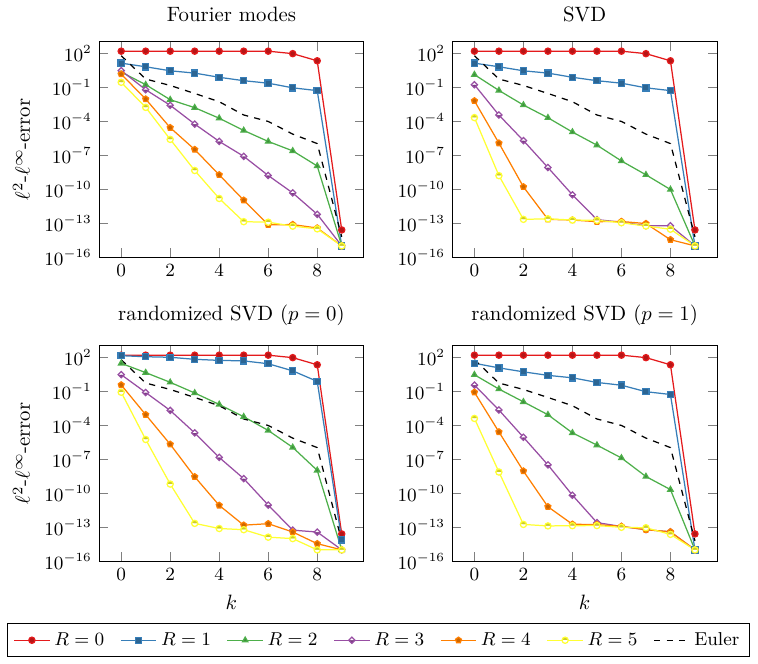}
    \end{center}
    \caption{%
        Experiment 2:
        Maximum $\ell^2$-errors over time vs.\ number of Parareal iterations for different choices and 
        ranks of $G_n$.
        The dashed lines indicate the Parareal errors when the $G_n$ are given by a single backward Euler
        step.
    }
    \label{fig:ex2_max_errs}
\end{figure}

We partition the time domain $[0,1]$ into 10 equally-sized
subintervals of 10 time steps each. Due to the time-dependence of the heat conductivity,
the resulting fine solvers $F_n'$ are no longer normal.  We plot the
decay of the singular values of each $F_n'$ in
\cref{fig:ex2_solution_svals} on the right.  We observe that the singular values
decay rapidly and almost identically for each time-interval $[t_{n-1},
  t_n]$.  Thus, we expect good performance of the low-rank coarse
solvers already for a small truncation rank $R$.  In
\cref{fig:ex2_max_errs}, we show the maximum $\ell^2$-error over time
with respect to the number of Parareal iterations $k$ and truncation ranks $R$
for different choices of coarse solvers.  The top-left plot shows the
error for a naive a~priori choice of the $G_n$ using the Fourier modes
$\sqrt{2}\cos(m\pi x)$.  We compare this choice to
SVD-based coarse solvers, where the SVD is either computed exactly
(top-right plot), or using a randomized SVD with no ($p = 0$) or a
single ($p = 1$) oversampling vector (bottom plots).  We now include
$b_n$ in the definitions of both the Fourier- and SVD-based coarse
solvers.

In all cases, for $R \geq 3$, Parareal with theses coarse solvers
converges significantly faster compared to Parareal with $G_n$ given
by a single Euler step.  We also observe that
using an SVD to define $G_n$ instead of using Fourier modes greatly
improves the convergence speed when higher-order modes ($R \geq 2$)
are used.  In particular, for $R = 5$ the SVD-based Parareal reaches
an error of $10^{-13}$ already for $k=2$, whereas with Fourier modes,
this error is only reached at $k=5$.  As computing an exact SVD of the
fine solvers $F_n'$ using standard methods is prohibitively expensive,
it is important to note that using a randomized SVD with only a single
oversampling vector already performs almost as well as using an exact
SVD.  Using no oversampling at all, however, the randomized SVD cannot
improve the convergence speed for $R \leq 2$ compared to using Fourier modes.

\begin{figure}[t]
    \begin{center}
        \includegraphics{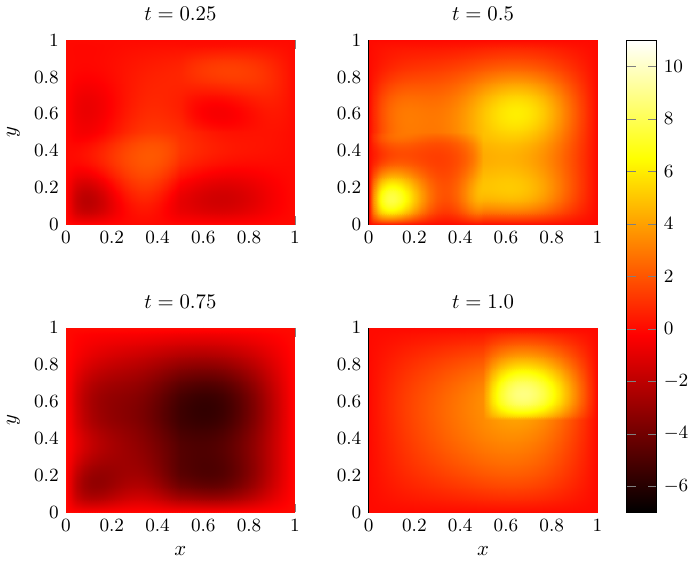}
    \end{center}
    \caption{%
        Experiment 3: Solution of test problem \cref{eq:ex3} at different times $t$.}
    \label{fig:ex3_solution}
\end{figure}

\subsection{Experiment 3}

We proceed with a computationally more challenging problem by considering the
two-dimensional heat equation given by
\begin{equation}
    \label{eq:ex3}
    \begin{aligned}
        u_t(x,y,t) - \nabla_{x,y} \cdot \left[ d(x,y,t) \nabla_{x,y} u(x,y)\right] &=  f(x,y,t) & (x,y) &\in \Omega, \\
        u(x, y, 0) &= u_0(x, y) & (x,y) &\in \Omega, \\
        u(x, y, t) &= 0 & (x,y) &\in \partial\Omega,
    \end{aligned}
\end{equation}
with $\Omega := (0,1)^2$, $t \in [0,1]$ and
\begin{align*}
    d(x,y,t) &= 1 + 0.9 \sin(7 \pi t) \chi_{[0,1/2]^2}(x,y) + 0.9 \cos(5 \pi t) \chi_{[1/2,1]^2}(x,y), \\
    f(x,y,t) &= 100\sin(5 \pi t)(1 + \cos(3\pi x))(1 + \sin(4\pi y)), \\
    u_0(x,y) &= 10\cdot\chi_{[0.6,0.8]^2}(x,y).
\end{align*}
For the discretization, we use a triangular mesh with $4\cdot 100 \cdot 100$ elements and $100$ time
steps ($\Delta t = 0.01$).
The solution at times $t = 0.25,\,0.5,\,0.75,\,1.0$ is visualized in \cref{fig:ex3_solution}.

\begin{figure}[t]
    \begin{center}
        \includegraphics{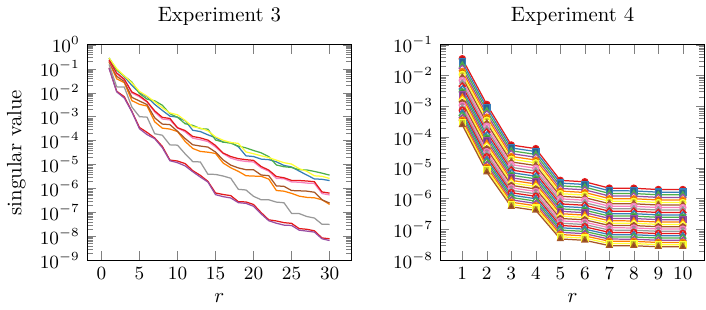}
    \end{center}
    \caption{%
      Singular values of the fine solvers $F_n'$
      (left: experiment~3, right: experiment~4).  }
    \label{fig:ex3_ex4_svals}
\end{figure}
\begin{figure}[t]
    \begin{center}
        \includegraphics{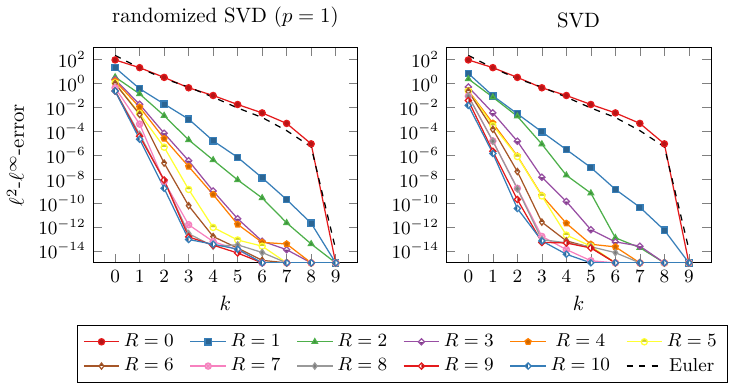}
    \end{center}
    \caption{%
        Experiment 3: Maximum $\ell^2$-errors over time vs.\ number of Parareal iterations for different
        ranks $R$ of $G_n$.
        The dashed lines indicate the Parareal errors when the $G_n$ are given by a single backward Euler
        step.
    }
    \label{fig:ex3_max_errs}
\end{figure}

As before, we consider 10 equally-sized time intervals of 10 time steps each. Compared to
experiment~2, we observe a slower decay of the singular values of the
fine solvers in \cref{fig:ex3_ex4_svals} (left), and the decay rate now
visibly depends on the selected time interval.  Still, the decay is
exponential and SVD-based coarse solvers of rank 1 are already
sufficient to outperform single Euler-step coarse solvers
(\cref{fig:ex3_max_errs}).
For 5 Parareal iterations, which require at least half the time to solve
the PDE sequentially in time, Parareal with single Euler-step coarse solvers
only reduces the error to $10^{-2}$.
Using an exact SVD with truncation rank $R=1$, an error of $10^{-7}$ is attained
for the same number of iterations.
With truncation rank $R=5$, an error of $10^{-6}$ is reached already for $k=2$, whereas with
single Euler-step solvers, this error is not attained until the fine solvers
have propagated through the entire time domain. When a randomized SVD
with a single oversampling vector ($p = 1$) is used instead of an exact SVD, convergence is
slightly slower.  $R=1$ is still sufficient to obtain a significantly
faster error decay compared to single Euler-step coarse solvers.  To
reach an error of $10^{-6}$ at $k=2$, the rank of the coarse solvers
has to be incremented by one to $R=6$.

\DeclareRobustCommand{\showsdirkmark}{\tikz \node[semithick]{\pgfuseplotmark{o}};}
\DeclareRobustCommand{\showsdirklmark}{\tikz \node[semithick]{\pgfuseplotmark{x}};}
\DeclareRobustCommand{\showsdirkntmark}{\tikz \node[semithick]{\pgfuseplotmark{triangle}};}
\DeclareRobustCommand{\showsdirklntmark}{\tikz \node[semithick]{\pgfuseplotmark{square}};}
\begin{figure}[t]
    \begin{center}
        \includegraphics{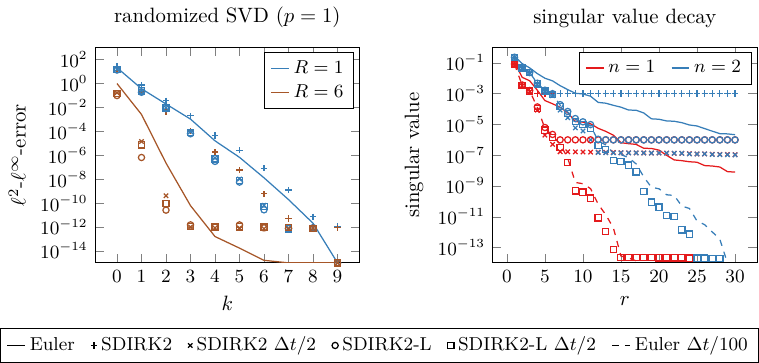}
    \end{center}
    \caption{%
        Experiment 3: Effect of the choice of time-stepping method on Parareal convergence;
        left: Maximum $\ell^2$-errors over time vs.\ number of Parareal iterations for different
        ranks $R$ of $G_n$, right: singular values of the fine solvers $F_n'$ for different time
        intervals $[t_{n-1}, t_n]$.
        The time-stepping method used is indicated by the line style and marker type. ``$\Delta t / 2\!$'' and
        ``$\Delta t / 100\!$'' indicate that the time-step size of the fine solver has been scaled by a
        factor of $1/2$ or $1/100$ compared to the standard time-step size for this experiment.
    }
    \label{fig:ex3_sdirk}
\end{figure}

\subsubsection{Effect of time discretization} \label{sec:effect_time_disc}

The effectiveness of our method depends on the rapid singular value decay of the fine solvers $F_n'$.
While such a decay is theoretically guaranteed for problems of parabolic type when $F_n'$ is the
analytical solution operator, the chosen time-stepping method might have a negative effect on the
singular value decay.
To study such a potential effect, we repeat experiment~3 for two different second-order Runge-Kutta
methods (SDIRK2, SDIRK2-L) with the following Butcher tableaus:
\[
    \text{SDIRK2: } \begin{array}{c|cc}
        1 & 1 &  \\
        0 & -1 & 1 \\
        \hline
        & 1/2 & 1/2
     \end{array}
    \qquad\qquad
    \text{SDIRK2-L: }  \begin{array}{c|cc}
        \gamma & \gamma &  \\
        1      & 1-\gamma & \gamma \\
        \hline
        & 1-\gamma & \gamma 
    \end{array} \quad \gamma = \frac{2-\sqrt{2}}{2}.
\]
Both methods belong to the class of singly-diagonally implicit Runge-Kutta (SDIRK) methods.
When using direct linear solvers, SDIRK-2 is attractive for solving linear ODE systems with
time-dependent operators, as it does not require any additional matrix factorizations compared to the
backward Euler method.
SDIRK-2 is not \mbox{L-stable}, however.
In contrast, the SDIRK2-L method is L-stable, but requires twice as many matrix factorizations.

For both time steppers, we consider the corresponding Parareal method with ran\-dom\-ized-SVD-based coarse
solvers. The convergence of these schemes is compared against the backward Euler method in the left
plot of \cref{fig:ex3_sdirk}.
In all cases, the error is computed with respect to the time-discrete solution obtained with the same time-stepping method.

We observe that for truncation rank $R=1$, all three time-stepping methods yield comparable convergence
rates, with SDIRK2-L showing the fastest convergence.
For $R=6$, however, the convergence rate for SDRIK2 stagnates compared to the other methods.
Indeed, for SDIRK2, the decay of the singular values of the corresponding fine solvers $F_n'$
(right plot of \cref{fig:ex3_sdirk}), stagnates already at a value of $10^{-3}$.
This effect can be explained by the lack of L-stability, which causes a bad approximation of solution
components associated with small singular values of the analytical solution operator.
For SDIRK2-L, we observe a much faster singular value decay, although a plateau with slow decay at
$10^{-7}$ is still present.
We attribute this plateau to the fact that, while being L-stable, the stability function $R(z)$ of the
SDIRK2-L method is not monotonically decreasing for $z \to -\infty$.
In contrast, we observe no such plateau for the backward Euler method, which has a monotonically
decreasing stability function.
However, being a first-order method, the overall approximation of the analytical solution operator is not
very good:
When considering a significantly smaller time-step size ($\Delta t / 100$), we observe a much faster
singular value decay of the $F_n'$, which is well-matched by SDIRK2-L before reaching the plateau.

Overall, we see that a bad approximation of the analytical solution operator by the chosen time-stepping
method can indeed affect the performance of our method.
In the case of experiment 3, halving the time step size (SDIRK2/SDIRK2-L with $\Delta t / 2$ in
\cref{fig:ex3_sdirk}) is already sufficient for obtaining similar convergence rates for all
considered time-stepping methods.

\begin{figure}[t]
    \includegraphics[width=0.49\textwidth]{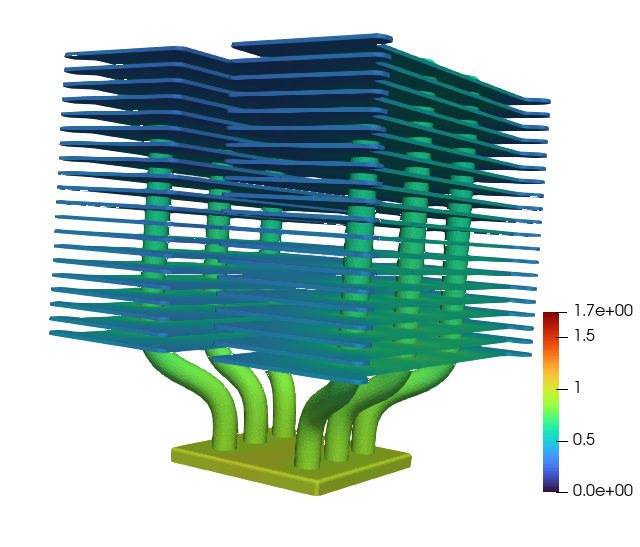}
    \includegraphics[width=0.49\textwidth]{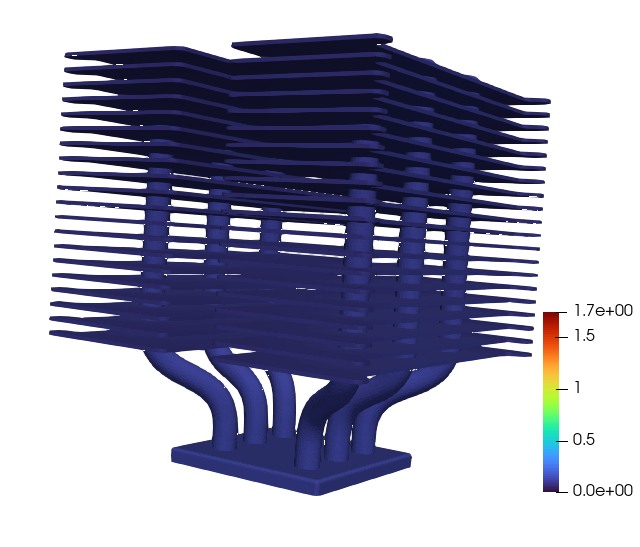}\\
    \includegraphics[width=0.49\textwidth]{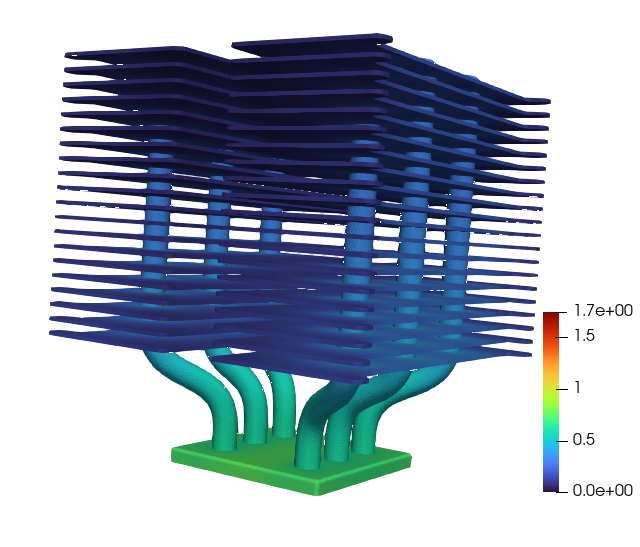}
    \includegraphics[width=0.49\textwidth]{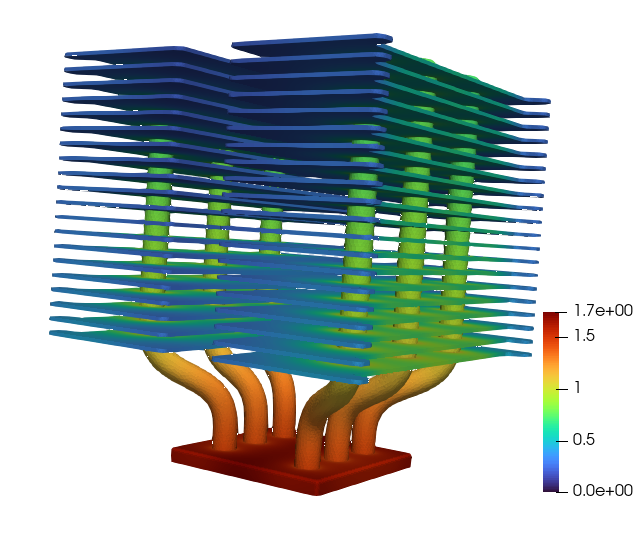}
    \caption{%
        Experiment 4: Solution of \cref{eq:ex4} at times $t=0.25$ (top left), $t=0.5$ (top right),
        $t = 0.75$ (bottom left) and $t= 1$ (bottom right).
    }
    \label{fig:ex4_solution}
\end{figure}

\subsection{Experiment 4} \label{sec:ex4}
Finally, we consider a larger three-dimensional test problem on the heat-sink geometry shown in
\cref{fig:ex4_solution}.
The geometry is scaled to fit into a cube with side length $1$.
We partition the domain $\Omega$ into subdomains corresponding to the fins
($\Omega_{\text{fin}}$), heat pipes ($\Omega_{\text{pipe}}$) and the baseplate ($\Omega_{\text{base}}$)
of the heat sink.
Further, we let $\Gamma_{\text{fin}}$ denote the boundary of the fins and $\Gamma_{\text{bot}}$ the
bottom boundary of the baseplate.
We solve the following heat equation with mixed boundary conditions:
\begin{equation}
    \label{eq:ex4}
    \begin{aligned}
        u_t(\mathbf{x},t) - \nabla \cdot \left[d(\mathbf{x}) \nabla_{\mathbf{x}} u(\mathbf{x})\right]
        &=  0 
        & \mathbf{x} &\in \Omega,\ t \in [0,1],\\
        -d(\mathbf{x}) \nabla_{\mathbf{x}} u(\mathbf{x},t) \cdot \mathbf{n}(\mathbf{x})
        &= (\frac{1}{2}+t) \cdot u(\mathbf{x},t)
        & \mathbf{x} &\in \Gamma_{\text{fin}},\\
        -d(\mathbf{x}) \nabla_{\mathbf{x}} u(\mathbf{x},t) \cdot \mathbf{n}(\mathbf{x})
        &= -g(t)
        & \mathbf{x} &\in \Gamma_{\text{bot}},\\
        -d(\mathbf{x}) \nabla_{\mathbf{x}} u(\mathbf{x},t) \cdot \mathbf{n}(\mathbf{x})
        &= 0
        & \mathbf{x} &\in \Gamma \setminus (\Gamma_{\text{fin}} \cup \Gamma_{\text{bot}}),\\
        u(\mathbf{x},0) &= 0.
    \end{aligned}
\end{equation}
Here, $d(\mathbf{x})$ is given by
\[
    d(\mathbf{x}) =
    \begin{cases}
        10 & \mathbf{x} \in \Omega_{\text{fin}}, \\
        100 & \mathbf{x} \in \Omega_{\text{base}}, \\
        1000 & \mathbf{x} \in \Omega_{\text{pipe}},
    \end{cases}
\]
accounting for different heat conductivities of the materials that
make up the heat-sink components.  We assume that the heat inflow at
the baseplate $g(t)$ is given by different load patterns of the form
\[
    g(t) =
    \begin{cases}
         50\cdot\frac{t}{0.3} & t \leq 0.3, \\
         50\cdot\left(1 + \operatorname{sign}\left(
            \sin\left(\frac{t-0.3}{0.3}\cdot 8\cdot \pi\right)\right)\right) & 0.3 < t \leq 0.6, \\
         50\cdot\left(1 + \cos\left(\frac{t-0.6}{0.4}\cdot 20\cdot \pi\right)\right) & 0.6 < t.
    \end{cases}
\]
The time-dependent Robin parameter $(1/2 + t)$ roughly models the
cooling of the heat sink by a fan that speeds up over time.

For the discretization in space, we use a tetrahedral mesh with 1,405,317
elements, yielding a finite-element space with 444,693 degrees of
freedom.  For the time discretization, we use 500 equidistant
backward Euler steps.  The solution at times $t = 0.25,\,0.5,\,0.75,\,1.0$ is
visualized in \cref{fig:ex4_solution}.

\begin{figure}
    \begin{center}
        \includegraphics{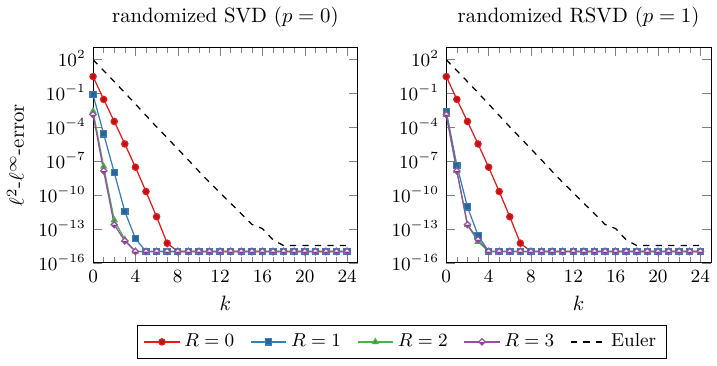} 
    \end{center}
    \caption{%
      Experiment 4: Maximum $\ell^2$-errors over time vs.\ number of
      Parareal iterations for different ranks $R$ of $G_n$.  The
      dashed lines indicate the Parareal errors when the $G_n$ are
      given by a single backward Euler step.  }
    \label{fig:ex4_max_errs}
\end{figure}

For the Parareal algorithm, we partition the time domain into 25 intervals of equal size
with 20 time steps each.  The
singular values of the fine solvers $F_n$ decay quickly, with the third singular value being
almost three orders of magnitude smaller than the first
(\cref{fig:ex3_ex4_svals}, right).  The singular values show a very
similar decay throughout all time intervals,
up to a factor that decreases as the time index $n$ and, hence,
the Robin parameter $(1/2 + t)$, increase.

We consider low-rank coarse solvers $G_n$ that are computed using randomized SVDs
with no ($p=0$) or one ($p = 1$) oversampling vector and compare their
performance with single Euler-step coarse solvers in
\cref{fig:ex4_max_errs}.  We observe that the Euler coarse solvers
perform very badly for this example, reaching machine precision only
at $k=18$, whereas using constant $G_n$ ($R=0$), machine precision is
reached already after half as many iterations.  Using low-rank coarse
solvers, we can further halve the number of required iterations when
we set the truncation rank $R$ to $2$ or higher.  When we are willing
to accept a slightly larger error of around $10^{-13}$, $k=2$ iterations
suffice, whereas for $k=1$, we can still reach an error of $10^{-7}$.
The use of oversampling in the SVD computation only has a relevant
impact for $R=1$, where the error can improve several orders of
magnitude.  For $R>1$, there is no significant difference.

For this experiment, we also report algorithm runtimes along with the
number of required (adjoint) fine solver evaluations in \cref{tab:runtimes}%
We see that the required time for the randomized SVDs does not
significantly increase for larger $R$.  This is due to the fact that
we have chosen a direct linear solver (\texttt{SuperLU}) for the time
stepping: in each time interval, we need to evaluate $F_n$ for $R+p+1$
initial values.  As all initial values are known at the same time, we
can perform the time stepping simultaneously by solving the same
linear system for $R+p+1$ right-hand sides in each time step.  Since
for direct solvers, the bulk of the computational time lies within
the LU decomposition of the system matrix, solving for additional
right-hand sides comes at only a very small additional cost.  After
having evaluated $F_n$, we can proceed in the same way with the $R+p$
evaluations of $F_n'^*$.

While using a direct linear solver might not be realistic in many
typical Parareal applications, we remark that also with iterative
solvers significant optimizations are possible, for instance by using
block Krylov methods or vectorized software implementations.
Additionally, all individual initial and terminal value problems are
completely independent, so their solution can be trivially
parallelized.

\begin{table}
    \caption{Experiment 4. Total runtimes in seconds (total number of
      fine solver and adjoint fine solver evaluations per time
      interval) for different coarse solvers. ``$G_n\!$'' refers to the
      computation of the rank-$R$ coarse solvers using randomized SVDs
      with oversampling parameter $p$.
      The required number of fine solver evaluations per time interval
      is given by $K_{F,n} + 1$, where $K_{F,n}$ is given by \cref{eq:num_F_eval}.
      One additional fine solver evaluation is needed
      to obtain an approximation at intermediate time points.
      The total time for a sequential
      solution is
    $3{,}290$
    seconds.
    }
    \label{tab:runtimes}
    \setlength{\tabcolsep}{0.1em}
    \centering
    \begin{filecontents*}{ex4_s25_times_and_solves.dat}
p r setup setupfine 0 0fine 1 1fine 2 2fine 3 3fine 4 4fine
$p=0$ $r=0$ 169 (1) 336 (2) 503 (3) 671 (4) 838 (5) 1005 (6)
$p=0$ $r=1$ 337 (3) 506 (4) 674 (5) 843 (6) 1013 (7) 1182 (8)
$p=0$ $r=2$ 339 (5) 507 (6) 676 (7) 844 (8) 1013 (9) 1182 (10)
$p=0$ $r=3$ 342 (7) 510 (8) 678 (9) 846 (10) 1014 (11) 1182 (12)
$p=1$ $r=0$ 335 (3) 502 (4) 669 (5) 837 (6) 1005 (7) 1174 (8)
$p=1$ $r=1$ 339 (5) 509 (6) 678 (7) 847 (8) 1017 (9) 1186 (10)
$p=1$ $r=2$ 342 (7) 509 (8) 677 (9) 845 (10) 1014 (11) 1183 (12)
$p=1$ $r=3$ 345 (9) 513 (10) 682 (11) 851 (12) 1020 (13) 1189 (14)
\end{filecontents*}
\pgfplotstabletypeset[
        every head row/.style={
            before row={
                \toprule
                &
                & 
                \multicolumn{2}{c}{$G_n$} &
                \multicolumn{2}{c}{$u^0$} &
                \multicolumn{2}{c}{$u^1$} &
                \multicolumn{2}{c}{$u^2$} &
                \multicolumn{2}{c}{$u^3$} &
                \multicolumn{2}{c}{$u^4$}\\
            },
            after row={
                \midrule
            },
            output empty row,
        },
        every last row/.style={
            after row=\bottomrule
        },
        columns/p/.style={
            string type,
        },
        columns/r/.style={
            column type={@{\extracolsep{0.3em}}r},
            string type,
        },
        columns/setup/.style={
            column type={@{\extracolsep{1.2em}}r},
        },
        columns/setupfine/.style={
            column type={@{\extracolsep{0.3em}}c},
            string type,
        },
        columns/0/.style={
            column type={@{\extracolsep{1.2em}}r},
            int detect,
        },
        columns/0fine/.style={
            column type={@{\extracolsep{0.3em}}c},
            string type,
        },
        columns/1/.style={
            column type={@{\extracolsep{1.2em}}r},
            int detect,
        },
        columns/1fine/.style={
            column type={@{\extracolsep{0.3em}}c},
            string type,
        },
        columns/2/.style={
            column type={@{\extracolsep{1.2em}}r},
        },
        columns/2fine/.style={
            column type={@{\extracolsep{0.3em}}c},
            string type,
        },
        columns/3/.style={
            column type={@{\extracolsep{1.2em}}r},
        },
        columns/3fine/.style={
            column type={@{\extracolsep{0.3em}}c},
            string type,
        },
        columns/4/.style={
            column type={@{\extracolsep{1.2em}}r},
        },
        columns/4fine/.style={
            column type={@{\extracolsep{0.3em}}c},
            string type,
        },
        every row no 4/.style={
            before row=\hline
        }
    ]{ex4_s25_times_and_solves.dat}
\end{table}

Finally, the number of required fine solver evaluations when using
low-rank coarse solvers is actually smaller than for the Euler coarse
solver due to the much higher Parareal convergence speed.  For
instance, to reach an error below $10^{-7}$, no more than 7 
evaluations of $F_n$ and $F_n^*$ per time interval are required, whereas with Euler coarse
solvers, 10 evaluations of $F_n$ are needed.

We plot the overall speedup factor achieved by the different Parareal
variants in comparison to a completely sequential solution of
\cref{eq:ex4} in \cref{fig:ex4_times}.  To achieve an error of
$10^{-7}$, an optimal speedup of 5 is achieved with $R \geq 1$ whereas
with Euler coarse solvers, only a speedup of around $1.8$ is possible.
For an error of $10^{-13}$, there is almost no speedup with the Euler
coarse solvers, whereas with low-rank coarse solvers, we can still
reach a speedup of around 4 for $R \geq 2$.

\def\showeulermark{\tikz \node[semithick]{\pgfuseplotmark{x}}; }
\begin{figure}
    \begin{center}
        \includegraphics{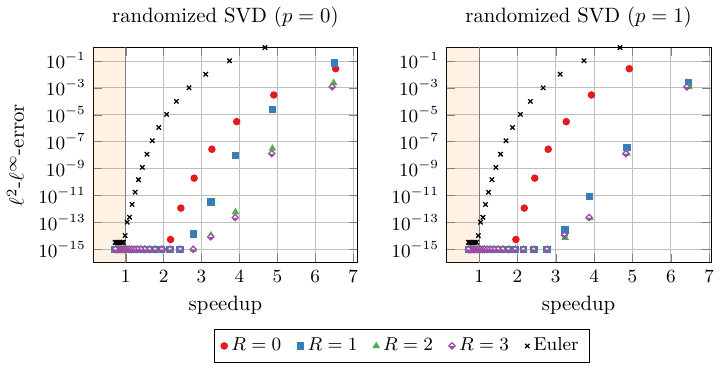} 
    \end{center}
    \caption{%
      Experiment 4: Maximum $\ell^2$-errors over time vs.\ speedup
      w.r.t.\ a sequential solution for different ranks $R$ of $G_n$
      and different numbers of Parareal iterations.  The marker
      \protect\showeulermark indicates the Parareal speedup when the
      $G_n$ are given by a single backward Euler step.  }
    \label{fig:ex4_times}
\end{figure}

\subsection{Error estimator efficiency}

For experiments 3 and 4, we evaluate the efficiencies $\eta^k$
of the a posteriori error estimator \cref{eq:apost_bound} given by
\begin{equation}
    \label{eq:def_efficiency}
    \eta^k := \frac{\max_{t \in [0,T]}\|u^k(t) - u(t)\|}{\max_{1 \leq n < N}\varepsilon\sum_{m=1}^{n-1} \delta^{n-m-1} \|u^k_m -
    u^{k-1}_m\|},
\end{equation}
where, as in the previous experiments, we consider the $L^{\infty}$-in-time
error by extending $u_{n-1}^k$ into $[t_{n-1}, t_n)$ via $F_{n}$.
The constants $\delta$ and $\varepsilon$ are obtained from the computed randomized SVDs,
where we use the additional oversampling vector ($p=1$) to approximate $\sigma_{n,R+1}$.
We plot $\eta^k$ for different values of $R$ in
\cref{fig:ex3_ex4_eff}.  We observe that in some cases, $\eta^k$ can
become larger than $1$, which means that the error is underestimated.
However, this only happens when the actual error is below $10^{-13}$
(dashed segments), so we attribute these cases to limited floating
point accuracy.  For both experiments, the efficiency of the error
estimator is quite good with $\eta^k > 0.1$ in most cases.  In
general, we cannot expect to be able to bound $\eta^k$ from below:
except for the special case in \cref{thm:diagonal_convergence}, the
actual convergence rate might be better than what is predicted by the
SVD truncation error, since dominant left-singular vectors of $F_n'$
might be projected onto less significant right-singular vectors of
$F_{n+1}'$.  Thus, it is not surprising that we observe cases where
$\eta_k \ll 0.1$.  However, since $\|u_n^{k+1}-u_n^k\| \leq
\|e_n^{k+1}\| + \|e_n^{k}\| \leq 2\|e_n^{k}\|$ when $\|e_n^{k+1}\| <
\|e_n^k\|$, we expect no more than one unnecessary Parareal
iteration in general.

\begin{figure}
    \begin{center}
        \includegraphics{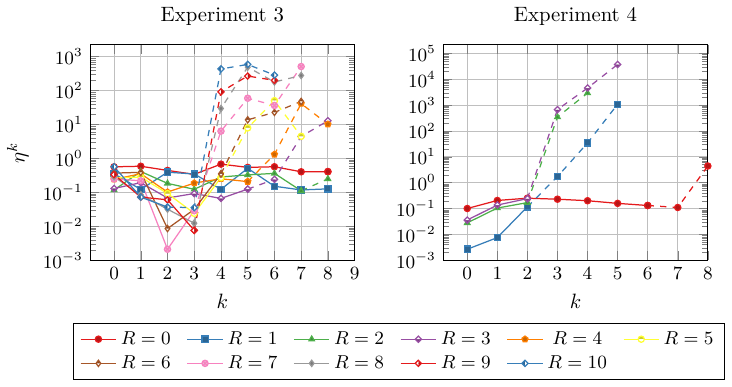}
    \end{center}
    \caption{
        Efficiencies $\eta^k$ \cref{eq:def_efficiency} for the a posteriori error bound
        \cref{eq:apost_bound} vs.\ number of Parareal iterations for different ranks $R$ of $G_n$, $p=1$.
        The plots are drawn dashed when the maximum error falls below $10^{-13}$.
    }
    \label{fig:ex3_ex4_eff}
\end{figure}

\section{Conclusions}
We have shown that using low-rank coarse solvers can dramatically accelerate the
convergence of Parareal for problems of parabolic type.
Thus, low-rank coarse solvers can increase the overall parallelism of Parareal by
trading global iterations for embarrassingly parallel fine solver evaluations in 
the algorithm's setup phase.
After the setup phase, low-rank coarse solvers are very fast to evaluate, which reduces
the runtime of the sequential coarse corrections and, hence, increases parallelism even further.
At the same time, as we have seen in experiment~4, low-rank coarse solvers can even lead to a reduction
of the overall computational work.

In this work we have focused on the solution of a single linear PDE.
While this approach cannot be expected to perform well for purely hyperbolic problems \cite{GanderPararealforHyperbolic},
an extension of the approach to parametric or non-linear problems
by applying further techniques from the reduced basis methodology seems promising.
For large-scale problems, the approximation of space- and time-localized transfer operators could
accelerate the setup phase of the algorithm by replacing expensive globally-coupled fine
solver evaluations with small independent time-evolution problems that could be solved on a single
compute node.

\bibliographystyle{plainurl}
\bibliography{paper}

\begin{thebibliography}{10}

\bibitem{MR4732743}
Katia Ait-Ameur and Yvon Maday.
\newblock Multi-step variant of the parareal algorithm: convergence analysis and numerics.
\newblock {\em ESAIM Math. Model. Numer. Anal.}, 58(2):673--694, 2024.
\newblock \href {https://doi.org/10.1051/m2an/2024014} {\path{doi:10.1051/m2an/2024014}}.

\bibitem{MR3701994}
Peter Benner, Mario Ohlberger, Anthony Patera, Gianluigi Rozza, and Karsten Urban, editors.
\newblock {\em Model reduction of parametrized systems}, volume~17 of {\em MS\&A. Modeling, Simulation and Applications}.
\newblock Springer, Cham, 2017.
\newblock Selected papers from the 3rd MoRePaS Conference held at the International School for Advanced Studies (SISSA), Trieste, October 13--16, 2015.
\newblock \href {https://doi.org/10.1007/978-3-319-58786-8} {\path{doi:10.1007/978-3-319-58786-8}}.

\bibitem{Buhr2020245}
Andreas Buhr, Laura Iapichino, Mario Ohlberger, Stephan Rave, Felix Schindler, and Kathrin Smetana.
\newblock Localized model reduction for parameterized problems.
\newblock In Peter Benner, Stefano {Grivet-Talocia}, Alfio Quarteroni, Gianluigi Rozza, Wilhelmus Schilders, and Lu{\'i}s~Miguel Silveira, editors, {\em Model {{Order Reduction}} ({{Volume}} 2)}. De Gruyter, Berlin, Boston, 2021.
\newblock \href {https://doi.org/10.1515/9783110671490-006} {\path{doi:10.1515/9783110671490-006}}.

\bibitem{MR3824169}
Andreas Buhr and Kathrin Smetana.
\newblock Randomized local model order reduction.
\newblock {\em SIAM J. Sci. Comput.}, 40(4):A2120--A2151, 2018.
\newblock \href {https://doi.org/10.1137/17M1138480} {\path{doi:10.1137/17M1138480}}.

\bibitem{carrel2023low}
Benjamin Carrel, Martin~J. Gander, and Bart Vandereycken.
\newblock Low-rank parareal: a low-rank parallel-in-time integrator.
\newblock {\em BIT Numerical Mathematics}, 63(1):13, 2023.
\newblock \href {https://doi.org/10.1007/s10543-023-00953-3} {\path{doi:10.1007/s10543-023-00953-3}}.

\bibitem{chen2014use}
Feng Chen, Jan~S. Hesthaven, and Xueyu Zhu.
\newblock On the use of reduced basis methods to accelerate and stabilize the parareal method.
\newblock {\em Reduced Order Methods for modeling and computational reduction}, pages 187--214, 2014.
\newblock \href {https://doi.org/10.1007/978-3-319-02090-7_7} {\path{doi:10.1007/978-3-319-02090-7_7}}.

\bibitem{DalcinFang2021Mpi4pyStatusUpdate}
Lisandro Dalcin and Yao-Lung~L. Fang.
\newblock Mpi4py: {{Status Update After}} 12 {{Years}} of {{Development}}.
\newblock {\em Computing in Science \& Engineering}, 23(4):47--54, 2021.
\newblock \href {https://doi.org/10.1109/MCSE.2021.3083216} {\path{doi:10.1109/MCSE.2021.3083216}}.

\bibitem{DobrevKolevEtAl2017TwoLevelConvergenceTheory}
V.~A. Dobrev, {\relax Tz}.~Kolev, N.~A. Petersson, and J.~B. Schroder.
\newblock Two-{{Level Convergence Theory}} for {{Multigrid Reduction}} in {{Time}} ({{MGRIT}}).
\newblock {\em SIAM Journal on Scientific Computing}, 39(5):S501--S527, 2017.
\newblock \href {https://doi.org/10.1137/16M1074096} {\path{doi:10.1137/16M1074096}}.

\bibitem{FarhatCortialEtAl2006TimeparallelImplicitIntegrators}
Charbel Farhat, Julien Cortial, Clim{\`e}ne Dastillung, and Henri Bavestrello.
\newblock Time-parallel implicit integrators for the near-real-time prediction of linear structural dynamic responses.
\newblock {\em International Journal for Numerical Methods in Engineering}, 67(5):697--724, 2006.
\newblock \href {https://doi.org/10.1002/nme.1653} {\path{doi:10.1002/nme.1653}}.

\bibitem{GanderPetcu2008AnalysisKrylovSubspace}
M.~Gander and M.~Petcu.
\newblock Analysis of a {{Krylov}} subspace enhanced parareal algorithm for linear problems.
\newblock {\em ESAIM: Proceedings}, 25:114--129, 2008.
\newblock \href {https://doi.org/10.1051/proc:082508} {\path{doi:10.1051/proc:082508}}.

\bibitem{GanderPararealforHyperbolic}
Martin~J. Gander.
\newblock Parareal for {{Hyperbolic Problems Just Does}} not {{Work}} -- {{Or Does}} it?
\newblock In {\em Proceedings of the 29th International Conference on Domain Decomposition Methods}, 2026.
\newblock To appear.
\newblock URL: \url{https://www.ddm.org/DD29/proceedings/ID35_pages.pdf}.

\bibitem{gander2008nonlinear}
Martin~J. Gander and Ernst Hairer.
\newblock Nonlinear convergence analysis for the parareal algorithm.
\newblock In {\em Domain decomposition methods in science and engineering XVII}, pages 45--56. Springer, 2008.
\newblock \href {https://doi.org/10.1007/978-3-540-75199-1_4} {\path{doi:10.1007/978-3-540-75199-1_4}}.

\bibitem{gander2014analysis}
Martin~J. Gander and Ernst Hairer.
\newblock Analysis for parareal algorithms applied to {H}amiltonian differential equations.
\newblock {\em Journal of Computational and Applied Mathematics}, 259:2--13, 2014.
\newblock \href {https://doi.org/10.1016/j.cam.2013.01.011} {\path{doi:10.1016/j.cam.2013.01.011}}.

\bibitem{GanderLunet2024}
Martin~J. Gander and Thibaut Lunet.
\newblock {\em Time Parallel Time Integration}.
\newblock SIAM, Philadelphia, PA, 2024.
\newblock \href {https://doi.org/10.1137/1.9781611978025} {\path{doi:10.1137/1.9781611978025}}.

\bibitem{MR4643841}
Martin~J. Gander, Thibaut Lunet, Daniel Ruprecht, and Robert Speck.
\newblock A unified analysis framework for iterative parallel-in-time algorithms.
\newblock {\em SIAM J. Sci. Comput.}, 45(5):A2275--A2303, 2023.
\newblock \href {https://doi.org/10.1137/22M1487163} {\path{doi:10.1137/22M1487163}}.

\bibitem{GanderOhlbergerEtAl2024PararealAlgorithmCoarse}
Martin~J. Gander, Mario Ohlberger, and Stephan Rave.
\newblock A {{Parareal}} algorithm without {{Coarse Propagator}}?, 2024.
\newblock \href {https://arxiv.org/abs/2409.02673} {\path{arXiv:2409.02673}}.

\bibitem{gander2007analysis}
Martin~J. Gander and Stefan Vandewalle.
\newblock Analysis of the parareal time-parallel time-integration method.
\newblock {\em SIAM Journal on Scientific Computing}, 29(2):556--578, 2007.
\newblock \href {https://doi.org/10.1137/05064607X} {\path{doi:10.1137/05064607X}}.

\bibitem{MR4234221}
Laura Grigori, Sever~A. Hirstoaga, Van-Thanh Nguyen, and Julien Salomon.
\newblock Reduced model-based parareal simulations of oscillatory singularly perturbed ordinary differential equations.
\newblock {\em J. Comput. Phys.}, 436:Paper No. 110282, 18, 2021.
\newblock \href {https://doi.org/10.1016/j.jcp.2021.110282} {\path{doi:10.1016/j.jcp.2021.110282}}.

\bibitem{GustafssonMcBain2020ScikitfemPythonPackage}
Tom Gustafsson and G.~D. McBain.
\newblock Scikit-fem: {{A Python}} package for finite element assembly.
\newblock {\em Journal of Open Source Software}, 5(52):2369, 2020.
\newblock \href {https://doi.org/10.21105/joss.02369} {\path{doi:10.21105/joss.02369}}.

\bibitem{MR2806637}
N.~Halko, P.~G. Martinsson, and J.~A. Tropp.
\newblock Finding structure with randomness: probabilistic algorithms for constructing approximate matrix decompositions.
\newblock {\em SIAM Rev.}, 53(2):217--288, 2011.
\newblock \href {https://doi.org/10.1137/090771806} {\path{doi:10.1137/090771806}}.

\bibitem{he2010reduced}
Liping He.
\newblock The reduced basis technique as a coarse solver for parareal in time simulations.
\newblock {\em Journal of Computational Mathematics}, pages 676--692, 2010.
\newblock URL: \url{https://www.jstor.org/stable/43693608}.

\bibitem{MR3408061}
Jan~S. Hesthaven, Gianluigi Rozza, and Benjamin Stamm.
\newblock {\em Certified reduced basis methods for parametrized partial differential equations}.
\newblock SpringerBriefs in Mathematics. Springer, Cham; BCAM Basque Center for Applied Mathematics, Bilbao, 2016.
\newblock BCAM SpringerBriefs.
\newblock \href {https://doi.org/10.1007/978-3-319-22470-1} {\path{doi:10.1007/978-3-319-22470-1}}.

\bibitem{MR4876550}
Bangti Jin, Qingle Lin, and Zhi Zhou.
\newblock Optimizing coarse propagators in parareal algorithms.
\newblock {\em SIAM J. Sci. Comput.}, 47(2):A735--A761, 2025.
\newblock \href {https://doi.org/10.1137/23M1619733} {\path{doi:10.1137/23M1619733}}.

\bibitem{LehoucqSorensenEtAl1998ARPACKUsersGuide}
R.~B. Lehoucq, D.~C. Sorensen, and C.~Yang.
\newblock {\em {{ARPACK Users}}' {{Guide}}}.
\newblock Software, {{Environments}}, and {{Tools}}. {Society for Industrial and Applied Mathematics}, 1998.
\newblock \href {https://doi.org/10.1137/1.9780898719628} {\path{doi:10.1137/1.9780898719628}}.

\bibitem{lions2001resolution}
Jacques-Louis Lions, Yvon Maday, and Gabriel Turinici.
\newblock R{\'e}solution d'edp par un sch{\'e}ma en temps parar{\'e}el.
\newblock {\em Comptes Rendus de l'Acad{\'e}mie des Sciences-Series I-Mathematics}, 332(7):661--668, 2001.
\newblock \href {https://doi.org/10.1016/S0764-4442(00)01793-6} {\path{doi:10.1016/S0764-4442(00)01793-6}}.

\bibitem{MartinssonTropp2020RandomizedNumericalLinear}
Per-Gunnar Martinsson and Joel~A. Tropp.
\newblock Randomized numerical linear algebra: {{Foundations}} and algorithms.
\newblock {\em Acta Numerica}, 29:403--572, 2020.
\newblock \href {https://doi.org/10.1017/S0962492920000021} {\path{doi:10.1017/S0962492920000021}}.

\bibitem{MilkRaveEtAl2016PyMORGenericAlgorithms}
R.~Milk, S.~Rave, and F.~Schindler.
\newblock {{pyMOR}} -- {{Generic Algorithms}} and {{Interfaces}} for {{Model Order Reduction}}.
\newblock {\em SIAM Journal of Scientific Computing}, 38(5):S194--S216, 2016.
\newblock \href {https://doi.org/10.1137/15M1026614} {\path{doi:10.1137/15M1026614}}.

\bibitem{MR3379913}
Alfio Quarteroni, Andrea Manzoni, and Federico Negri.
\newblock {\em Reduced basis methods for partial differential equations}, volume~92 of {\em Unitext}.
\newblock Springer, Cham, 2016.
\newblock \href {https://doi.org/10.1007/978-3-319-15431-2} {\path{doi:10.1007/978-3-319-15431-2}}.

\bibitem{ZenodoRecord}
Stephan Rave.
\newblock A parareal algorithm with low-rank coarse solvers (supplement).
\newblock Zenodo, 2026.
\newblock \href {https://doi.org/10.5281/zenodo.20442179} {\path{doi:10.5281/zenodo.20442179}}.

\bibitem{MR4401794}
Julia Schleu{\ss} and Kathrin Smetana.
\newblock Optimal local approximation spaces for parabolic problems.
\newblock {\em Multiscale Model. Simul.}, 20(1):551--582, 2022.
\newblock \href {https://doi.org/10.1137/20M1384294} {\path{doi:10.1137/20M1384294}}.

\bibitem{MR4589122}
Julia Schleu\ss, Kathrin Smetana, and Lukas Ter~Maat.
\newblock Randomized quasi-optimal local approximation spaces in time.
\newblock {\em SIAM J. Sci. Comput.}, 45(3):A1066--A1096, 2023.
\newblock \href {https://doi.org/10.1137/22M1481002} {\path{doi:10.1137/22M1481002}}.

\bibitem{2020SciPy-NMeth}
Pauli Virtanen, Ralf Gommers, Travis~E. Oliphant, et~al.
\newblock {{SciPy} 1.0: Fundamental Algorithms for Scientific Computing in Python}.
\newblock {\em Nature Methods}, 17:261--272, 2020.
\newblock \href {https://doi.org/10.1038/s41592-019-0686-2} {\path{doi:10.1038/s41592-019-0686-2}}.

\bibitem{MR3803287}
Shu-Lin Wu.
\newblock Toward parallel coarse grid correction for the parareal algorithm.
\newblock {\em SIAM J. Sci. Comput.}, 40(3):A1446--A1472, 2018.
\newblock \href {https://doi.org/10.1137/17M1141102} {\path{doi:10.1137/17M1141102}}.

\end{thebibliography}

\end{document}